\documentclass[11pt,reqno]{amsart}
\usepackage{graphicx, fullpage,bbm, subfigure}
\usepackage{amssymb, amsmath, amsthm,tikz, bm,bbm}
\usepackage{epstopdf,enumitem,esint}
\usepackage{verbatim,mathabx}
\usepackage{xcolor,mathrsfs,mathtools,stmaryrd,microtype}
\usepackage[T1]{fontenc}
\usepackage[utf8]{inputenc}
\usepackage{hyperref}
\newcommand{\be}{\begin{equation} }
\newcommand{\ee}{\end{equation}}
\newcommand{\bse}{\begin{subequations}}
\newcommand{\ese}{\end{subequations}}

\newcommand{\mr}{\mathbb{R}}

\newcommand{\p}{\partial}

\theoremstyle{plain}
\newtheorem{theorem}{Theorem}[section]

\newtheorem{lemma}{Lemma}[section]
\newtheorem{proposition}{Proposition}[section]
\newtheorem{definition}{Definition}[section]
\theoremstyle{remark}
\newtheorem{remark}{Remark}[section]
\allowdisplaybreaks[4]

\theoremstyle{definition}

\numberwithin{equation}{section}

\title[Spectral stability in the modified Camassa-Holm equation]{Spectral stability in the modified Camassa-Holm equation}

\author[L. Fan]{Lili Fan}%
\address[Lili Fan]{School of Mathematics and Statistics,
Henan Normal University, Xinxiang 453007, China}
\email{fanlili@htu.edu.cn}
\author[H. Gao]{Hongjun Gao}%
\address[Hongjun Gao]{School of Mathematics,
Southeast University, Nanjing 211189, China}
\email{hjgao@seu.edu.cn}
\author[J. Li]{Ji Li*}%
\address[Ji Li]{School of Mathematics and Statistics,
Huazhong University of Science and Technology, Wuhan 430074, China}
\email{liji@hust.edu.cn (corresponding author)}
\date{}

\begin{document}

\begin{abstract}
We investigate the spectral stability of small-amplitude, periodic, traveling-wave solutions of
the modified Camassa-Holm equation with cubic nonlinearities. More precisely, we analyze the $L^2(\mr)$-spectrum of the associated linearized operator in a neighborhood of the origin in the spectral plane. Inspired by a recently novel method based on Kato's perturbation theory [Berti et al, Full description of Benjamin-Feir instability of Stokes waves in deep water, \textit{Invent. Math.}, 230 (2022), 651-711.], we provide a complete description of the spectrum near the origin of the linearized operator--an integro-differential operator with periodic coefficients--and thus prove that such waves are not subject to modulational instability. Moreover, a spectral analysis reveals a remarkable threshold phenomenon: such waves with wave number $k^2\leq 3$ exhibit spectral stability, while instability emerges when $k^2>3$.


\end{abstract}

\thispagestyle{empty}
\maketitle

\setcounter{tocdepth}{1}




\section{Introduction}
In this paper, we are concerned with the spectral stability of periodic traveling waves of the following
modified Camassa–Holm (mCH) equation
\begin{equation}\label{mch}
\left(1-\partial_x^2\right) u_t+\left(u^3-u^2u_{x x}-u u^2_x+ u^2_x u_{x x}\right)_x = 0.
\end{equation}
This equation has been proposed in the modelling of nonlinear water waves by Fokas \cite{fokas1995}, and recently been shown to govern weakly nonlinear shallow water waves of moderate amplitude propagating over a flat bottom \cite{chen2022}, where the function $u(t, x)$ is related to the horizontal velocity in certain level of fluid. Structurally, the mCH equation \eqref{mch} is formally integrable in the sense that it admits a bi-Hamiltonian structure \cite{olver1996} and was later shown to admit a Lax formulation \cite{qiao2006}. Moreover, it is worth mentioning that the scaling limit case of equation \eqref{mch}, when combined with the first-order term $\gamma u_x$, is the \textit{short-pulse} equation for $v=u_x$
\[
v_{t}=\frac 1 3 \left(v^3\right)_{xx}+\gamma v,
\]
which is a model for the propagation of ultra-short light pulses in silica optical fibers \cite{schafer2004}.

Traveling wave solutions are fundamental for understanding the dynamics of nonlinear PDEs. These solutions not only reveal intrinsic pattern formation mechanisms but also provide bridges between theoretical models and observed phenomena in fields such as hydrodynamic turbulence, nonlinear optics, and quantum field theory. Among classical traveling wave types, solitary waves and periodic traveling waves have received particular attention. Research on the mCH equation, especially concerning the stability of its solitary wave solutions, has developed into such an extensive research domain that a complete survey lies beyond the scope of this work. For foundational aspects of this field, we refer readers to \cite{li2021,qu2013}, with recent advances documented in \cite{chen2023,li2024,li2025,li2025-1}. Notably, while stability of solitary waves in the mCH equation has been extensively studied,  to our knowledge, the stability of periodic traveling waves has not been studied in the literature.

In this paper, we aim at understanding how a periodic traveling wave behaves under smooth localized perturbations. Specifically, we will study the existence of small-amplitude periodic traveling waves of the mCH equation \eqref{mch}, as well as determine the modulational stability/instability of the said waves. Here, modulational instability, also known as Benjamin-Feir instability, refers to spectral instability of the underlying wave to perturbations in $L^2(\mr)$ in a sufficiently small ball about the origin in the spectral plane (see Definition \ref{def3.1}, Section \ref{sec spectrum}). The investigation of this problem dates back to the pioneering work of Benjamin and Feir \cite{1967PRSL,1967JFM1} in the 1960s, who discovered that small-amplitude spatially periodic Stokes waves exhibit instability under long-wavelength perturbations. Subsequent three decades saw this phenomenon corroborated by both experimental and numerical evidence. While rigorous mathematical proof emerged later: Bridges and Mielke \cite{Br1995} first established modulational instability for finite-depth waves (1995), and Nguyen and Strauss \cite{2023CPAM} resolved the infinite-depth case (2020). Most recently, Berti et al. \cite{2022INVENT,2023ARMAFINITE,2024CMP} fully described the figure-8 structure of unstable eigenvalues, covering deep, finite, and critical water depths.

Beyond the classical Stokes wave, modulational instability has been widely observed in various other nonlinear waves, including: quasiperiodic solutions of the nonlinear Schrödinger (NLS) equation \cite{1984PRL}, the periodic standing waves of the NLS equation \cite{2021JNS}, the periodic standing wave solutions of the generalized Korteweg–de Vries (KdV) equation \cite{2010ARMA}, and the wave trains of the Camassa-Holm equation \cite{Joh2024}. For the small-amplitude periodic traveling waves under consideration, significant results exist for a variety of dispersive models including the fractional KdV \cite{johnson2013}, the Whitham equation \cite{2015SAM}, the general Benjamin-Bona-Mahony \cite{HP,H1}, the full-dispersion Camassa-Holm equations \cite{HP1}, the generalized KdV equation \cite{audiard2024}, the Ostrovsky equation \cite{bhavna2024}, the Novikov equation \cite{2024Novikov}, and the $b$-family equation \cite{fan2024}. Particularly notable is recent work on generalized KdV equations \cite{ma2024}, revealing figure-8 configuration of unstable eigenvalues near the origin in the spectral plane. This paper aims at extending spectral stability/instability analysis in \cite{2022INVENT,ma2024} to the mCH model \eqref{mch}, addressing its nonlocal and strongly nonlinear nature.

\subsection{Main results}
Although the basic idea of the approach stems from the significant works \cite{2022INVENT,ma2024}, the nonlocal dispersion and strong nonlinearity inherent to the mCH equation \eqref{mch} makes the spectral computation a lot more involved. On the one hand, the presence of an integro-differential operator requires the development of a new symplectic basis, distinct from those employed in \cite{2022INVENT,ma2024}. On the other hand, the complexity of the linearized operator leads to computational difficulties in determining the entries of the matrix $\mathrm{A}_{\xi, a}$, where we use Mathematica to carry out the arduous calculations. Despite these obstacles, we establish the following main result concerning spectral stability of small-amplitude periodic traveling waves in the mCH equation.
\begin{theorem}[Modulational Stability vs. Spectral Stability/Instability]\label{the1.1}
A sufficiently small $2 \pi / k$-periodic traveling wave solution of the modified Camassa-Holm equation \eqref{mch}, as constructed in Lemma \ref{lem2.1}, is (\textrm{i}) modulationally stable for any $k>0$; (\textrm{ii}) spectrally stable if $k^2\leq3$; (\textrm{iii}) spectrally unstable if $k^2>3$.
\end{theorem}

The above theorem demonstrates that the mCH equation does not exhibit the modulational instability for small-amplitude periodic traveling waves. This result deviates from the typical scenario in which a critical wave number threshold $k^*> 0$ exists such that periodic traveling waves of period $2\pi/k$ are modulationally stable when $k \in (0,k^*)$ and modulationally unstable when $k> k^*$. However, this threshold behavior is not universal. Contrary examples cover from the shallow water equations of the KdV equation \cite{johnson2013}, the Camassa-Holm equation and Degasperis-Procesi \cite{fan2024}, to the fluid system of electronic Euler-Poisson \cite{noble2023}, where all small-amplitude periodic traveling waves were proven to be modulationally stable. These findings suggest that modulational instability arises either higher-order dispersive or strong nonlinear effects are considered. This also implies that modulational instability analyses must a priori consider the potential absence of a critical wave number threshold. Moreover, our findings align well with the spectral stability observed in the context of mCH solitary waves (see \cite{li2021} for details) for $k\rightarrow 0$.

The remainder of this paper is organized as follows. In Section \ref{sec expression}, we construct the existence of a family of the one-dimensional small-amplitude periodic traveling waves of the mCH equation by employing implicit function theorem and Lyapunov-Schmidt reduction, and give a parameterization of these waves. In Section \ref{sec spectrum}, we formulate the spectral stability problem based on the linearization of \eqref{mch} about the obtained periodic traveling wave, and present the complete spectral Theorem \ref{the6.1}. In Section \ref{sec basis}, we focus on the introduction of Kato's perturbed basis and the reduction to the study of the eigenvalues of a $3 \times 3$ matrix $\mathrm{A}_{\xi, a}$. In Section \ref{sec Matrix}, we compute the matrix $\mathrm{A}_{\xi, a}$ for $|\xi|, |a|$ sufficiently small. In Section \ref{sec Block}, we carry out a non-perturbative block-diagonalization that conjugates the matrix $\mathrm{A}_{\xi, a}$ to a new matrix with two decoupled blocks: a purely imaginary number, and a $2 \times 2$ block, whose eigenvalues can be computed explicitly. Finally, in Section \ref{sec non-modulational}, we give the spectral instability analysis far from origin.
\subsection{Notations}
Throughout this paper, we will use the following notations. The space $L^2(\mathbb{R})$ denotes the set of real or complex-valued, Lebesgue measurable functions over $\mathbb{R}$ such that
\[
\|f\|_{L^2(\mathbb{R})}=\left(\int_{\mathbb{R}}|f(x)|^2 \mathrm{~d} x\right)^{1 / 2}<+\infty,
\]
and $L^2(\mathbb{T})$ denotes the space of $2 \pi$-periodic, measurable, real or complex-valued functions over $\mathbb{R}$ such that
\[
\| f \|_{L^2(\mathbb{T})}=\left(\frac{1}{\pi} \int_{-\pi}^{\pi}|f(x)|^2 \mathrm{~d} x\right)^{1 / 2}<+\infty,
\]
and $\|f\|_{L^{\infty}(\mathbb{T})}:=\operatorname{ess} \sup _{0<z \leq 2\pi}|f(z)|<\infty$ if $p=\infty$. Let $H^1(\mathbb{T})$ consist of $L^2(\mathbb{T})$ functions whose derivative is in $L^2(\mathbb{T})$. Let $H^{\infty}(\mathbb{T})=\bigcap_{k=0}^{\infty} H^k(\mathbb{T})$.
For $f \in L^1(\mathbb{T})$, the Fourier series of $f$ is defined by
$$
\sum_{n \in \mathbb{Z}} \widehat{f}_n \mathrm{e}^{\mathrm{i} n z}, \quad \text { where } \widehat{f}_n=\frac{1}{ \pi} \int_{-\pi}^{\pi} f(z) \mathrm{e}^{-\mathrm{i} n z} \mathrm{~d} z .
$$
If $f \in L^2(\mathbb{T})$, then its Fourier series converges to $f$ pointwise almost everywhere. We define the $L^2(\mathbb{T})$-inner product as
\begin{equation}\label{1.2}
\left( f, g\right)=\frac{1}{\pi} \int_{-\pi}^{\pi} f(z) \bar{g}(z) \mathrm{d} z=\sum_{n \in \mathbb{Z}} \widehat{f}_n \overline{\widehat{g}}_n.
\end{equation}

We denote by $\mathcal{O}\left(\xi^{m_1} a^{n_1}, \ldots, \xi^{m_p} a^{n_p}\right), m_j, n_j \in \mathbb{N}$, analytic functions of ($\xi, a$) with values in a Banach space $X$ which satisfy, for some $C>0$, the bound $\left\|\mathcal{O}\left(\xi^{m_j} a^{n_j}\right)\right\|_X \leq C \sum_{j=1}^p|\xi|^{m_j}|a|^{n_j}$ for small values of ( $\xi, a$ ). We denote $r_k\left(\xi^{m_1} a^{n_1}, \ldots, \xi^{m_p} a^{n_p}\right)$ scalar functions $\mathcal{O}\left(\xi^{m_1} a^{n_1}, \ldots, \xi^{m_p} a^{n_p}\right)$ which are also real analytic.
\section{Asymptotically small-amplitude periodic traveling waves}\label{sec expression}
We consider traveling waves of the mCH equation \eqref{mch}, which take the form
\[
u(x, t) = u(x - ct),
\]
where $c>0$ is the speed of propagation, and $u$ satisfies the ODE
\begin{equation}\label{2.1}
-c\left(u^{\prime}-u^{\prime \prime \prime}\right)+\left(u^{3}-u^2 u^{\prime\prime}-u (u^{\prime})^2+ (u^{\prime})^2u^{\prime \prime}\right)^{\prime}=0.
\end{equation}
By elementary ODE theory, we necessarily have $u\in C^{\infty}$ provided that $u^2<c$ as
\[
\left(u^2-c- (u^{\prime})^2\right)(u-u^{\prime\prime})^{\prime}+2 u^{\prime}(u-u^{\prime\prime})^{2}=0.
\]
Integrating \eqref{2.1} once,  we arrive at the following second order ODE
\begin{equation}\label{2.2}
-c u+c u^{\prime \prime}+u^{3}-u^2u^{\prime\prime}-u\left(u^{\prime}\right)^2+\left(u^{\prime}\right)^2u^{\prime\prime}=-2 b^3
\end{equation}
for some $b\in\mr$. Given the absence of scaling and Galilean invariance in this equation, we may not simply assume that $c=1, m=0$. Let $u$ be a $2 \pi / k$-periodic function of its argument, for some wave number $k>0$. Then, $w(z):=u(x)$ with $z=k x$, is a $2 \pi$-periodic function in $z$, satisfying
\begin{equation}\label{2.3}
- c w+ck^2 w_{z z}+ w^3-k^2 w^2 w_{z z}-k^2 w w_z^2+k^4 w_z^2 w_{z z}=-2 b^3.
\end{equation}
Define $F: H^2(\mathbb{T})\times\mr_{+}\times\mr_{+}\times\mr\rightarrow L^2(\mathbb{T})$ as
\begin{equation}\label{2.4}
F(w ; k, c, b)=- c w+ck^2 w_{z z}+ w^3-k^2 w^2 w_{z z}-k^2 w w_z^2+k^4 w_z^2 w_{z z}+2 b^3.
\end{equation}
We seek a non-constant $2\pi$-periodic solution $w \in H^2(\mathbb{T})$ of
\begin{equation}\label{2.5}
F(w ; k, c, b) = 0.
\end{equation}
Observing that \eqref{2.4} remains invariant under $z\mapsto z+z_0$, $z\mapsto-z$ for any $z_0\in\mr$, we may assume that $w$ is even. Clearly $F$ is analytic on its arguments. Then for any $k>0, c > 0$, suppose $w_0(c, b)$ makes a constant solution of \eqref{2.4}-\eqref{2.4}. It follows from the implicit function theorem that if non-constant solutions of \eqref{2.5} (and hence \eqref{2.3}) bifurcate from $w = w_0$ for some $c = c_0$, then necessarily,
\begin{equation*}
L_0 := \p_w F\left(w_0 ; c_0, k, b\right)=-c_0+k^2\left(c_0-w_0^2\right) \partial_z^2+3 w_0^2: H^2(\mathbb{T}) \rightarrow L^2(\mathbb{T})
\end{equation*}
is not an isomorphism. Further calculation reveals that $L_0 e^{inz}=0, n\in\mathbb{Z}$, if and only if
$c_0 =\frac{k^2n^2+3}{k^2n^2+1}w_0^2.$ Without loss of generality, we restrict our attention to $|n| = 1$ and get
\begin{equation}\label{2.6}
c_0 =\frac{k^2+3}{k^2+1}w_0^2.
\end{equation}
Form \eqref{2.4}-\eqref{2.6}, we get
\begin{equation}\label{2.7}
 w_0=(1+k^2)^{1/3} b, \qquad c_0=\frac{k^2+3}{k^2+1}w_0^2=\frac{k^2+3}{(k^2+1)^{1/3}}b^2.
\end{equation}
In this case, it is straightforward to verify that for any $k>0, b\neq 0$, the kernel of $L_0$ : $H^2(\mathbb{T}) \rightarrow L^2(\mathbb{T})$ is two-dimensional and spanned by $e^{\pm iz}$. Moreover, the co-kernel of $L_0$ is two-dimensional. Therefore, $L_0$ is a Fredholm operator of index zero. One may then follow an idea similar to that of \cite{HP,HP1} to employ a Lyapunov-Schmidt reduction and construct a one parameter family of non-constant, even and smooth solutions of \eqref{2.1} near $w=w_0(k, b)$ and $c=c_0(k, b)$. The small-amplitude expansion of these solutions is given as follows and the details are provided in Appendix \ref{app A}.
\begin{lemma}[Existence]\label{lem2.1}
For each $k > 0, b \neq 0$, there exists a family of $2\pi/k$-periodic traveling waves of \eqref{mch}
\begin{equation}\label{2.8}
w(k, a, b) := u\left(k\left(x-c(k, a, b) t\right)\right)
\end{equation}
for $|a|$ sufficiently small; $w$ and $c$ depend analytically on $k$ and $a$,
$w$ is smooth, even and $2\pi$-periodic in $z$, and $c$ is even in $a$. Furthermore, as $a \to 0$,
\begin{align}
w(z; k, a, b)= &\ w_0+ a \cos z+a^2 \left(A_0+A_2 \cos 2z\right)+\mathcal{O}\left(a^3\right), \label{2.9} \\
c(k, a, b) = &\ c_0 + a^2 c_2+\mathcal{O}\left(a^3 \right), \label{2.10}
\end{align}
with
\begin{equation}\label{2.11}
\displaystyle A_0=-\frac{\left(1+k^2\right)^{5/3}}{8 b k^2},\quad A_2=-2A_0=\frac{\left(1+k^2\right)^{5/3}}{4 b k^2}=\frac{\left(1+k^2\right)^2}{4k^2w_0},\quad
\displaystyle c_2=\frac{5+k^2}{4},
\end{equation}
and $w_0$, $c_0$ given in \eqref{2.7}.
\end{lemma}
\section{The complete modulational spectrum}\label{sec spectrum}
Performing linearization of the mCH equation \eqref{mch} about its one-dimensional periodic traveling wave solution $w$ given in \eqref{2.9}-\eqref{2.11}, and considering the perturbations to $w$ in the form $w+\varepsilon v(t,z)$, we obtain the linearized equation
\begin{equation*}
\begin{aligned}
&\left(1-k^2\p_z^2\right)\left(v_t-k c v_z\right)+ k\left(3 w^2 v- k^2\left(2 w w_{zz} v+w^2 v_{zz}\right) \right.\\
&\left.- k^2\left(2 w w_z v_{z}+w^2_z v\right)+k^4 \left( w_z^2v_{zz}+2w_zw_{zz}v_z\right)\right)_z=0.
\end{aligned}
\end{equation*}
For $v(z, t)=\mathrm{e}^{\lambda kt} V(z)$, $\lambda\in\mathbb{C}$, we have
\begin{equation}\label{3.1}
v_t=\mathcal{A}(k, a) v=\mathcal{J} \mathcal{L}(k, a) v,
\end{equation}
with $\mathcal{J}=\partial_z(1-k^2 \p_z^2)^{-1}$ and
\begin{equation}\label{3.2}
\mathcal{L}(k, a)=-k^2\partial_{z}\left(\left(c-w^2+k^2w_{z}^2\right) \partial_{z}\right)+c-3 w^2+k^2w_{z}^2+2 k^2 w w_{z z},
\end{equation}
where the notation of $a$ of the operator in \eqref{3.2} reflects the dependence of $w, c$ on the expansion parameter $a$ of Lemma \ref{lem2.1}. The operator $\mathcal{J}$ is skew-adjoint and $\mathcal{L}(k, a)$ is self-adjoint.
\begin{definition}[Spectral stability]\label{def3.1}
For a $2 \pi / k$-periodic traveling wave solution $u(x, t)=w(k(x-c t))$ of \eqref{mch}
where $w$ and $c$ are as in \eqref{2.9}-\eqref{2.11}, we say that the periodic wave $w$ is spectrally unstable if the $L^2(\mathbb{R})$-spectrum of the operator $\mathcal{A}(k, a)$ intersects the open, right half plane of $\mathbb{C}$. Otherwise, $w$ is deemed to be spectrally stable.
\end{definition}
Due to the fact that the operator $\mathcal{A}(k, a)$ is Hamiltonian and $w$ is an even function, the spectrum of $\mathcal{A}(k, a)$ is symmetric with respect to both the real and imaginary axes. Then $w$ is spectrally unstable if and only if the $L^2(\mathbb{R})$-spectrum of $\mathcal{A}(k, a)$ is not contained in the imaginary axis.

Since the coefficients of $\mathcal{A}_a$ are periodic functions, using Floquet theory, all solutions of \eqref{3.1} in $L^2(\mathbb{R})$ are of the form $V(z)=\mathrm{e}^{i \xi z} \tilde{V}(z)$ where $\xi \in\left(-\frac{1}{2}, \frac{1}{2}\right]$ is the Floquet exponent and $\tilde{V}$ is a $2 \pi$-periodic function; see \cite{H1} for a similar situation. Consequently, the spectral examination of the operator $\mathcal{A}_a$ in $L^2(\mathbb{R})$ is transformed to the study of a $\xi$-parameterized family of Bloch operators in $L^2(\mathbb{T})$. We present the precise reformulation in the following lemma.
\begin{lemma}\label{lem3.1}
Consider the operator $\mathcal{A}_a$ in \eqref{3.1} acting in $L^2(\mathbb{R})$, and the Bloch operators
\begin{equation}\label{A_xi}
\begin{aligned}
\mathcal{A}_{\xi}(k, a)
=&\left(\partial_z+i \xi\right)\left( 1-k^2\left(\partial_z+i \xi\right)^2\right)^{-1} \\
&\left(
-k^2\left(\partial_z+i \xi\right)\left(\left(c-w^2+k^2w_{z}^2\right) \left(\partial_z+i \xi\right)\right)+c-3 w^2+k^2w_{z}^2+2 k^2 w w_{z z}\right)\\
=&:\mathcal{J}_{\xi}\mathcal{L}_{\xi, a}
\end{aligned}
\end{equation}
acting in $L^2(\mathbb{T})$ with domain $H^1(\mathbb{T})$. Then
$$
\sigma\left(\mathcal{A}(k,a)\right)=\bigcup_{\xi \in\left(-\frac{1}{2}, \frac{1}{2}\right]} \sigma\left(\mathcal{A}_\xi(k,a)\right),
$$
here $\sigma(\mathcal{A})$ denotes the spectrum of an operator $\mathcal{A}$.
\end{lemma}
We now establish the symmetry property of the operator $\mathcal{A}_{\xi,a}$, which is useful in the subsequent discussion.
\begin{lemma}[Symmetry property]\label{lem-symmetry}
Assume that $\xi \in\left(-\frac{1}{2}, \frac{1}{2}\right]$. Then the spectra of the Bloch operators $\mathcal{A}_\xi(k,a)$ acting in $L^2(\mathbb{T})$ satisfy
\begin{equation}\label{3.3}
\sigma\left(\mathcal{A}_\xi(k,a)\right)=\overline{\sigma\left(\mathcal{A}_{-\xi}(k,a)\right)}
=-\sigma\left(\mathcal{A}_{-\xi}(k,a)\right)=-\overline{\sigma\left(\mathcal{A}_\xi(k,a)\right)}.
\end{equation}
\end{lemma}
\begin{proof} We consider $\mathcal{S}$ to be the reflection through the imaginary axis defined as follows
\begin{equation}\label{s}
\mathcal{S}\psi(z)=\overline{\psi(-z)},
\end{equation}
and notice that $\mathcal{A}_{\xi,a}$ anti-commutes with $\mathcal{S}$,
\begin{equation}\label{AS}
\left(\mathcal{A}_\xi(k,a)\mathcal{S}\psi\right)(z)=\mathcal{A}_\xi(k,a)\left(\overline{\psi(-z)}\right)
=-\overline{\left(\mathcal{A}_\xi(k,a)\psi\right)(-z)}=-\left(\mathcal{S}\mathcal{A}_\xi(k,a)\psi\right)(z),
\end{equation}
where we have used the fact that $w$ is an even function. Assume $\mu$ is the eigenvalue of $\mathcal{A}_\xi(k,a)$ with an associated eigenvector $\varphi$,
\[
\mathcal{A}_\xi(k,a)\varphi=\mu\varphi,
\]
then we have
\[
\mathcal{A}_\xi(k,a)\mathcal{S}\varphi=-\mathcal{S}\mathcal{A}_\xi(k,a)\varphi=-\overline{\mu}\mathcal{S}\varphi.
\]
Consequently, $-\overline{\mu}$ is an eigenvalue of $\mathcal{A}_\xi(k,a)$.

Consider $\mathcal{R}$ to be the reflection as follows
\[
\mathcal{R}\psi(z)=\psi(-z),
\]
then we have
\[
\left(\mathcal{A}_\xi(k,a)\mathcal{R}\right)\psi(z)=\mathcal{A}_\xi(k,a)\left(\psi(-z)\right)
=-\left(\mathcal{A}_{-\xi}(k,a)\psi\right)(-z)
=-\left(\mathcal{R}\mathcal{A}_{-\xi}(k,a)\psi\right)(z).
\]
This completes the proof.
\end{proof}
The above lemma tells us that the spectrum of $\mathcal{A}_\xi(k,a)$ is symmetric with respect to the imaginary axis, and it suffices to consider $\xi \in\left[0, \frac{1}{2}\right]$. Since $k$ is fixed, in what follows, we denote $\mathcal{A}_{\xi, a}=\mathcal{A}_{\xi}(k, a)$ and $\mathcal{L}_{\xi, a}=\mathcal{L}_{\xi}(k, a)$. According to \cite{2022INVENT}, the operator $\mathcal{A}_{\xi, a}$ is termed \textit{reversible} as \eqref{AS} and similarly, $\mathcal{L}_{\xi, a}$ is termed \textit{reversibility preserving}.

There emerge three distinct perturbation regimes according to the Floquet exponent $\xi$: same period perturbations, when $\xi=0$; long-wavelength perturbations, when $0<\xi\ll1$ and finite wavelength perturbations, otherwise. Long-wavelength perturbations furnish the spectral information of $\mathcal{A}$ in the vicinity of the origin in $\mathbb{C}$; see \cite{2015SAM}, for details. Hence, we say that $w$ is \emph{modulational stable} if the $L^2(\mathbb{T})$ spectra of $\mathcal{A}_{\xi, a}$ lie on the imaginary axis near the origin for $0<\xi\ll1$, and failure of it is known as \emph{modulational unstable}.
\subsection{Spectrum of $\mathcal{A}_{\xi, 0}$}
The spectral analysis for the operators $\mathcal{A}_{\xi, a}$ employs perturbation techniques that treat $\mathcal{A}_{\xi, a}$ as a perturbation of the constant-coefficient operator $\mathcal{A}_{\xi, 0}$. Indeed, one can establish the following estimate
\begin{equation}\label{A per}
\left\|\mathcal{A}_{\xi, a}-\mathcal{A}_{\xi, 0}\right\|_{H^1(\mathbb{T})\rightarrow L^2(\mathbb{T})}=\mathcal{O}(|a|), \quad \text{as}\;a\rightarrow 0.
\end{equation}
Here in the case $a=0$, the linear operator $\mathcal{A}_{\xi, 0}$ is
\begin{equation}\label{A dec}
\mathcal{A}_{\xi, 0}=J_{\xi} \mathcal{L}_{\xi, 0},
\end{equation}
with
$$
J_{\xi} := \left(\partial_z+i \xi\right)\left(1-k^2\left(\partial_z+i \xi\right)^2\right)^{-1}, \;\text{and}
\quad \mathcal{L}_{\xi, 0}:=-k^2\left(c_0-w_0^2\right)\left(\partial_z+i \xi\right)^2+c_0-3w_0^2.
$$
To locate the spectrum of $\mathcal{A}_{\xi, a}$, it is necessary to first determine the spectrum of $\mathcal{A}_{\xi, 0}$. A straightforward Fourier calculation shows that
\begin{subequations}\label{ev A nonper}
\begin{equation}\label{ev A FS}
\mathcal{A}_{\xi, 0} \mathrm{e}^{i n z}=i \omega_{n,\xi} \mathrm{e}^{i n z}, \quad n \in \mathbb{Z},
\end{equation}
where
\begin{equation}\label{ev A FS r}
\omega_{n,\xi} =\frac{2 k^2 w_0^2}{1+k^2}\frac{n+\xi}{1+k^2 (n+\xi)^2}\left((n+\xi)^2-1\right).
\end{equation}
\end{subequations}
Then the constant solutions of \eqref{mch} is spectrally stable to square integrable perturbations. Additionally, for sufficiently small $|a|$, Lemma \ref{lem-symmetry} implies that multiple eigenvalues
colliding on imaginary axis may bifurcate to leave the imaginary axis and induce instabilities though
simple eigenvalues remain purely imaginary for small $|a|$. Therefore, it is now important to determine the more precise locations of these eigenvalues $i\omega_{n,\xi}$, with particular emphasis on their multiplicities.

Given the Hamiltonian nature of  $\mathcal{A}_{\xi, a}$, we utilize the Krein signature in standard linear Hamiltonian theory to detect the possible collisions  (a detailed discussion appears in \cite[Proposition 7.1.14]{KP13}). Specifically, the Krein signature $K_n$ of an eigenvalue $i\omega_{n, \xi}$ of $\mathcal{A}_{\xi, 0}$ is defined as
\begin{equation}\label{3.16}
K_n :=\operatorname{sgn}\left(\left(\mathcal{L}_{\xi, 0} \mathrm{e}^{i n z}, \mathrm{e}^{i n z}\right)\right)
=\operatorname{sgn}\left((n+\xi)^2-1\right).
\end{equation}
A necessary condition for a pair of eigenvalues to leave imaginary axis after collision is that they carry opposite Krein signatures. In the case $\xi = 0$, it is obvious that the Krein signatures of all eigenvalues maintain positive with the sole exception of the zero eigenvalue at $n=\pm1$, where the Krein signature calculation becomes inconclusive. For $\xi\neq 0$, a straightforward analysis reveals that the eigenvalues $i\omega_{-1, \xi}$ and $i\omega_{0, \xi}$ possess negative Krein signatures, whereas all remaining eigenvalues exhibit positive Krein signatures. We summarize in the following lemma all the possible collisions of the eigenvalues $i\omega_{n, \xi}$.
\begin{lemma}\label{lem3.3}
For $\xi\in[0,\frac 1 2]$, the eigenvalues $i\omega_{n, \xi}$ have the following properties:
\begin{enumerate}[label=\textup{(\roman*)}]
\item For $\xi=0$, the collision occurs only  at $\omega_{0,0}=\omega_{-1,0}=\omega_{1,0}=0$;
\item For some $\xi_0\in(0,\frac 1 2]$, there exist possible collisions of eigenvalues $i\omega_{-1, \xi_0}$ and $i\omega_{n, \xi_0}, n\geq 1$, and of $i\omega_{0, \xi_0}$ and $i\omega_{-n, \xi_0}, n\geq 2$. Furthermore, if $\xi_0\in(0, 1-\frac {\sqrt{3}} 3]$, $i\omega_{-1, \xi_0}$ collides only possibly with $i\omega_{1, \xi_0}$.
 \item For $\xi\in(0,\frac 1 2]$, there holds
 \begin{itemize}
\item for $k^2\leq3$,
\begin{equation}\label{eigen}
\cdots<\omega_{-3, \xi}<\omega_{-2, \xi}<\omega_{0, \xi}<0<\omega_{-1, \xi}<\omega_{1, \xi}<\omega_{2, \xi}<\omega_{3, \xi}<\cdots;
\end{equation}

\item for $3<k^2\leq 4$, $\omega_{-1, \xi_0}=\omega_{1, \xi_0}$ with $\xi_0=\sqrt{1-\frac 3 {k^2}}$;
\item for $k^2\geq 4$, $\omega_{-2, \xi_0}=\omega_{0, \xi_0}$ with $\xi_0=1-\sqrt{1-\frac 3 {k^2}}$.
\end{itemize}
\end{enumerate}
\end{lemma}
\begin{proof}
At $\xi=0$, we find that
\[
\cdots<\omega_{-3,0}<\omega_{-2,0}<\omega_{0,0}=\omega_{-1,0}=\omega_{1,0}=0<\omega_{2,0}<\omega_{3,0}<\cdots.
\]

For $\xi\in(0,\frac 1 2]$, it is obvious
 \begin{itemize}
   \item $\omega_{-1, \xi}=\frac{2 k^2 w_0^2}{1+k^2}\frac{\xi-1}{1+k^2 (\xi-1)^2}\left((\xi-1)^2-1\right)>0$ and $\omega_{0, \xi}=\frac{2 k^2 w_0^2}{1+k^2}\frac{\xi}{1+k^2 \xi^2}\left(\xi^2-1\right)$<0.
  \item $\omega_{-n, \xi}<0,n\geq 2$ and $\omega_{n, \xi}>0,n\geq 1$.
 \end{itemize}
Consequently, two distinct types of eigenvalue collisions may occur: of $i\omega_{0, \xi}$ and $i\omega_{-n, \xi}$, $n\geq 2$, and of $\omega_{-1, \xi}$ and $\omega_{n, \xi}>0$, $n\geq 1$. On the other hand, the monotonic characteristic of the function $f(x)=\frac {x(x^2-1)} {1+k^2x^2}$ for $x^2>\frac 1 3$ permits
\[
\cdots<\omega_{-3, \xi}<\omega_{-2, \xi}<0<\omega_{1, \xi}<\omega_{2, \xi}<\omega_{3, \xi}<\cdots.
\]
Additionally, rigorous calculation establishes that when $(\xi-1)^2\in[\frac 1 3, 1)$
\begin{equation*}
0<\omega_{-1, \xi}\leq\frac{2 k^2 w_0^2}{1+k^2} \cdot \frac{2\sqrt{3}}{9+3 k^2}
<\frac{2 k^2 w_0^2}{1+k^2} \cdot \frac{6}{1+4 k^2}
<\omega_{2, \xi}=\frac{2 k^2 w_0^2}{1+k^2} \frac{(\xi+2)\left((\xi+2)^2-1\right)}{1+k^2(\xi+2)^2}.
\end{equation*}
Thus $i\omega_{-1, \xi}$ collides possibly only with $i\omega_{1, \xi}$ if $\xi_0\in(0, 1-\frac {\sqrt{3}} 3]$. Besides, if $k^2\leq3$, then for $\xi\in(0,\frac 1 2]$
\[
\omega_{-2, \xi}-\omega_{0, \xi}=-\frac{2 k^2 w_0^2}{1+k^2}\cdot\frac{2 (\xi -1)^2 \left(k^2 \xi ^2-2 k^2 \xi +3\right)}{\left(k^2 (\xi -2)^2+1\right) \left(k^2 \xi ^2+1\right)}<0,
\]
and
\[
\omega_{-1, \xi}-\omega_{1, \xi}=-\frac{2 k^2 w_0^2}{1+k^2}\cdot\frac{2 \xi ^2 \left(k^2 \xi ^2-k^2+3\right)}{k^4 \left(\xi ^2-1\right)^2+2 k^2 \left(\xi ^2+1\right)+1}<0,
\]
which, when combined with the foregoing arguments, gives the conclusions (iii).
\end{proof}
\begin{remark}\label{rem3.1}
Any of the eigenvalue collisions delineated in Lemma \ref{lem3.3} may potentially give rise to unstable eigenvalues for the operator $\mathcal{A}_{\xi, a}$. In the parameter regime $k^2\leq3$, spectral instability can only emerge from $\omega_{0,0}=\omega_{-1,0}=\omega_{1,0}=0$. Therefore, if modulational instability is precluded for $\xi\ll 1$, the associated wave must necessarily exhibit spectral stability.
\end{remark}
\subsection{Spectrum of $\mathcal{A}_{0, a}$}\label{sec3.2}
For $|a|$ small but $\xi=0$, zero is shown to be a generalized eigenvalue of $\mathcal{A}_{0, a}$ in $L^2(\mathbb{T})$, characterized by algebraic multiplicity three and geometric multiplicity two. Indeed, following the idea in \cite[Lemma 3.1]{2015SAM}, we differentiate \eqref{2.3} with respect to $z, a, b$ to get that
\begin{equation*}
\mathcal{A}_{0,a}\left(\partial_z w\right)=0, \;\; \mathcal{A}_{0,a}\left(\partial_a w\right)=-\left(\partial_a c\right)\left(\partial_z w\right), \;\; \mathcal{A}_{0,a}\left(\partial_b w\right)=-\left(\partial_b c\right)\left(\partial_z w\right),
\end{equation*}
respectively. Then
\begin{equation*}
\begin{aligned}
\varphi_1(z):=&\frac{(k^2+1)^{1/3}}{2(k^2+3)b}\left(\left(\partial_b c\right)\left(\partial_a w\right)-\left(\partial_a c\right)\left(\partial_b w\right)\right)\\
=&\cos z+2a A_2 \cos(2z)+2a\left(A_0-\frac{(k^2+1)^{2/3}}{2(k^2+3)b}c_2\right)+\mathcal{O}(a^2),\\
\varphi_2(z):=&-\frac 1 a \partial_z w=\sin z+2a A_2 \sin(2z)+\mathcal{O}(a^2)
\end{aligned}
\end{equation*}
are eigenvectors, whereas
$$
\varphi_3(z):=\frac 1 {\sqrt{2}(1+k^2)^{1/3}}\left(\partial_b w\right)=\frac 1 {\sqrt{2}}\left(1+\frac{\left(1+k^2\right)^{4/3}}{8 b^2 k^2}-\frac{\left(1+k^2\right)^{4/3}}{4 b^2 k^2}\cos(2z)\right)+\mathcal{O}(a^3)
$$
is a generalized eigenvector.
\subsection{Complete spectral result of $\mathcal{A}_{\xi, a}$}\label{sec3.3}
Our complete spectral result is the following.
\begin{theorem}\label{the6.1} For sufficiently small $|a|\ll 1$, $0<\xi\ll 1$, the operator $\mathcal{A}_{\xi, a}: \mathcal{V}_{\xi, a} \rightarrow \mathcal{V}_{\xi, a}$ can be represented by the $3 \times 3$ block-diagonal matrix
\begin{equation}\label{6.27}
\left(\begin{array}{c|c}
\mathrm{U} & \verb"0" \\
\hline \verb"0"^{\dagger} & i\mathrm{g}
\end{array}\right),
\end{equation}
where the matrix $\mathrm{U}$ expands as
\begin{equation}\label{6.25}
\mathrm{U}:=2i \xi\left(\frac {\mathrm{e}_{22}}{k^2+1}+r_0\left(a^2, \xi a, \xi^2\right)\right)+\xi\left(\begin{array}{cc}
i \left(a^2 r_1(1)+\xi^2 r_2(1)\right) & \mathrm{e}_b+r_3(\xi, a) \\
\frac{\mathrm{e}_w}{k^2+1} a^2\left(1+r_4(\xi, a)\right)-\mathrm{e}_b \xi^2\left(1+r_5(\xi, a)\right) & i \xi^2 r_6(1)
\end{array}\right),
\end{equation}
with $\mathrm{e}_{22}, \mathrm{e}_b, \mathrm{e}_w$ in \eqref{**}, \eqref{6.18}, \eqref{6.11} and the number $\mathrm{g}$ is
\begin{equation}\label{6.26}
\mathrm{g}:=-\xi\frac {2k^2}{1+k^2}w_0^2+\xi r(a^2,\xi^2),
\end{equation}
and $\verb"0"=(0,0)^{\dagger}$. The eigenvalues of $\mathrm{U}$ have the form
\begin{equation}\label{6.28}
\lambda_1^{ \pm}(\xi, a)=2i \xi\left(\frac {\mathrm{e}_{22}}{k^2+1}+r_0\left(a^2, \xi a, \xi^2\right)\right) \pm \xi \sqrt{\frac{\mathrm{e}_b\mathrm{e}_w}{k^2+1} a^2\left(1+\hat{r}_1(a, \xi)\right)-\mathrm{e}_b^2 \xi^2\left(1+\hat{r}_2(a, \xi)\right)},
\end{equation}
with
\[
\mathrm{e}_b\mathrm{e}_w=-\frac{\left(k^2-3\right)^2}{\left(k^2+1\right)^{2}}w_0^2=-\frac{b^2 \left(k^2-3\right)^2}{\left(k^2+1\right)^{4/3}}<0,\qquad \text{for}\;k^2\neq 3.
\]
\end{theorem}
\section{Kato's perturbed basis}\label{sec basis}
To study the modulational stability/instability of the wave $w$, we need to locate the spectrum of the operator $\mathcal{A}_{\xi, a}$ for  $\xi, |a|$ sufficiently small. Note that the operator $\mathcal{A}_{\xi, a}$ is a perturbation of $\mathcal{A}_{0, 0}$ with
$$
\left\|\mathcal{A}_{\xi, a}-\mathcal{A}_{0, 0}\right\|_{H^1(\mathbb{T})\rightarrow L^2(\mathbb{T})}=\mathcal{O}(|a|+\xi)
$$
uniformly in operator norm as $|a|,\xi\rightarrow 0$. The application of standard perturbation theory ensures spectral continuity between $\mathcal{A}_{\xi, a}$ and $\mathcal{A}_{0, 0}$. This continuity, together with the findings in Lemma \ref{lem3.3} (i), generates a well-defined spectral separation
$$
\sigma\left(\mathcal{A}_{\xi, a}\right) = \sigma_1\left(\mathcal{A}_{\xi, a}\right) \bigcup \sigma_2\left(\mathcal{A}_{\xi, a}\right), \quad \sigma_1\left(\mathcal{A}_{\xi, a}\right) \bigcap \sigma_2\left(\mathcal{A}_{\xi, a}\right)=\emptyset.
$$
The set $\sigma_1\left(\mathcal{A}_{\xi, a}\right)$ contains exactly three eigenvalues that emerge as continuous perturbations, when $|a|$ is small, from the eigenvalues $i\omega_{0, \xi}$, $i\omega_{ \pm 1, \xi}$ of $\mathcal{A}_{\xi, 0}$, whereas $\sigma_2\left(\mathcal{A}_{\xi, a}\right)$ corresponds to the spectral continuation of all eigenvalues $i \omega_{n, \xi}, |n| \geqslant 2$.

We denote by $\mathcal{V}_{\xi, a}$ the spectral subspace associated with $\sigma_1\left(\mathcal{A}_{\xi, a}\right)$, which is a three-dimensional $\mathcal{A}_{\xi, a}$-invariant subspace. Building upon the approach introduced in \cite{2022INVENT}, we utilize Kato's similarity transformation \cite[Chapter Two, §4]{kato2013}
\begin{equation}\label{4.1}
U_{\xi, a}:=\left(\operatorname{Id}-\left(P_{\xi, a}-P_{0,0}\right)^2\right)^{-1 / 2}\left[P_{\xi, a} P_{0,0}+\left(\operatorname{Id}-P_{\xi, a}\right)\left(\operatorname{Id}-P_{0,0}\right)\right]
\end{equation}
to construct a basis of $\mathcal{V}_{\xi, a}$. Here,
\begin{equation}\label{4.2}
P_{\xi, a}:=-\frac{1}{2 \pi \mathrm{i}} \oint_{\Gamma}\left(\mathcal{A}_{\xi, a}-\lambda\right)^{-1} \mathrm{~d} \lambda: L^2(\mathbb{T}) \rightarrow H^1(\mathbb{T})
\end{equation}
is the spectral projection, where $\Gamma$ belongs to the resolvent set of the operator $\mathcal{A}_{\xi, a}: H^1(\mathbb{T}) \rightarrow L^2(\mathbb{T})$ defined in \eqref{A_xi}. The analyticity of $\mathcal{A}_{\xi, a}$ implies that $P_{\xi, a}$ and $U_{\xi, a}$ are analytic in $(\xi, a)$ near $(0,0)$. Set
\[
\mathcal{V}_{\xi, a}:=\operatorname{Rg}\left(P_{\xi, a}\right).
\]
We will need the following properties:
\begin{align}
&\mathcal{A}_{\xi, a}: \mathcal{V}_{\xi, a} \rightarrow \mathcal{V}_{\xi, a},\label{4.3}\\
& \sigma\left(\mathcal{A}_{\xi, a}\right) \cap\{z \in \mathbb{C} \text { inside } \Gamma\}=\sigma\left(\mathcal{A}_{\xi, a} \mid_{ \mathcal{V}_{\xi, a}}\right)=\sigma_1\left(\mathcal{A}_{\xi, a}\right),\label{4.4} \\
&\mathcal{V}_{\xi, a}=U_{\xi, a} \mathcal{V}_{0,0},\label{4.5}\\
&U_{0, a}\; \text{is a real operator and}\; \dot{U}_{0, a}:=\left.\partial_\xi U_{\xi, a}\right|_{\xi=0} \text { are purely imaginary},\label{*}
\end{align}
whose proof is available in \cite{2022INVENT,ma2024}.

Due to the appearance the nonlocal operator $\mathcal{J}_\xi$, we introduce an alternative of symplectic basis different from \cite{2022INVENT,ma2024}. To this end, we first denote, for $\xi \neq 0$ the inverse operator
$$
\mathcal{E}_\xi:=\mathcal{J}_\xi^{-1}: H^1(\mathbb{T}) \rightarrow L^2(\mathbb{T}), \quad \xi>0
$$
which acts as a Fourier multiplier of symbol $\frac{1+k^2(\xi+\mu)^2}{i(\xi+\mu)}$. When $\xi=0$, the operator $\mathcal{J}_0=\partial_z(1-k^2 \p_z^2)^{-1}$ is invertible in the subspace
$$
H_0^s(\mathbb{T}):=\left\{u \in H^s(\mathbb{T}): \int_{\mathbb{T}} u(x) \mathrm{d} x=0\right\}.
$$
Hence, we define
$$
\mathcal{E}_0=\mathcal{J}_0^{-1}: H_0^1(\mathbb{T}) \rightarrow L_0^2(\mathbb{T}).
$$
\begin{definition} [$\xi$-symplectic and reversible basis]\label{def4.1} A linearly independent set $\left\{f_1^{+}, f_1^{-}, f_0^{+}\right\}$ is
\begin{enumerate}
\item[(i)] $\xi$-symplectic, $\xi \neq 0$, if one has
\begin{equation}\label{4.6}
\begin{aligned}
&\left(\mathcal{E}_\xi f_1^{+}, f_1^{+}\right)=\left(\mathcal{E}_\xi f_1^{-}, f_1^{-}\right)=i \xi  \left(\frac{1}{1-\xi ^2}-k^2\right), \quad\left(\mathcal{E}_\xi f_0^+, f_0^+\right)=-i \left(\frac{1}{\xi }+k^2 \xi\right),\\
&\left(\mathcal{E}_\xi f_1^{+}, f_1^{-}\right)=-\left(\mathcal{E}_\xi f_1^{-}, f_1^{+}\right)=\frac{1}{1-\xi ^2}+k^2, \quad\left(\mathcal{E}_\xi f_0^+, f_1^{+}\right)=\left(\mathcal{E}_\xi f_0^+, f_1^{-}\right)=0,
\end{aligned}
\end{equation}

and $0$-symplectic, if both $f_1^{+}, f_1^{-}\in L_0^2(\mathbb{T})$, $f_0\in \mathrm{Ker}\mathcal{J}_0=\mathrm{span}\langle1\rangle$ and
\begin{equation}\label{4.7}
\left(\mathcal{E}_0 f_1^{+}, f_1^{-}\right)=-\left(\mathcal{E}_0 f_1^{-}, f_1^{-}\right)=1+k^2, \quad\left(\mathcal{E}_0 f_1^{+}, f_1^{+}\right)=\left(\mathcal{E}_0 f_1^{-}, f_1^{-}\right)=0.
\end{equation}

\item[(ii)] reversible if $\mathcal{S}f_1^{+}=f_1^{+}, \mathcal{S}f_1^{-}=-f_1^{-}, \mathcal{S}f_0^+=f_0^+$, and $\mathcal{S}$ is given in \eqref{s}.
\end{enumerate}
\end{definition}
\begin{remark}\label{rem4.1} The basis $f_1^{+}=\cos z, f_1^{-}=\sin z, f^+_0=\frac 1 {\sqrt 2}$ is $\xi$-symplectic for any $\xi\in [0, \frac 1 2]$.
\end{remark}

We denote by $even(z)$ a real $2\pi$-periodic function which is even in $x$, and by $ odd(z)$ a real $2\pi$-periodic function which is odd in $z$. According to Definition \ref{def4.1}, it is easily to outline the following property of a reversible basis.
\begin{lemma}\label{lem4.1} Let $\{f_1^+, f_1^- , f_0^+\}$ be a reversible basis according to Definition \ref{def4.1}. Then
\begin{equation}\label{4.8}
f_1^+(z) = even(z) + i \, odd(z), \quad f_1^-(z) = odd(z) + i \, even(z), \quad f_0^+(z) = even(z) + i \, odd(z).
\end{equation}
\end{lemma}

The property \eqref{4.3} allows us to reduce the spectral analysis of $\mathcal{A}_{\xi, a}$ to that of a $3 \times 3$ matrix
\begin{equation}\label{A}
\mathrm{A}_{\xi, a}:=\left(\left(\tilde{f}_m^\sigma, \mathcal{A}_{\xi, a} f_{m^{\prime}}^{\sigma^{\prime}}\right)\right)_{m, m^{\prime}=0,1}^{\sigma, \sigma^{\prime}= \pm},
\end{equation}
where $\tilde{f}_m^\sigma$ for $(m, \sigma) \in\{(1, \pm),(0,+)\}$ denote dual basis of $\mathcal{A}_{\xi, a}$. We now focus on the representation of $\mathrm{A}_{\xi, a}$, for which the construction of an appropriate dual basis becomes essential. Due to the Hamiltonian symmetry inherited by $\mathcal{A}_{\xi, a}$,  the generalized eigenspace of the adjoint operator $\mathcal{A}^*_{\xi, a}$ is spanned by $\{\mathcal{E}_\xi f_{m}^{\sigma}\}_{m=0,1}^{\sigma= \pm}$. Assume that
\[
\begin{pmatrix}
\tilde{f}_1^+ \\
\tilde{f}_1^- \\
\tilde{f}_0^+
\end{pmatrix} =\begin{pmatrix}
\alpha_1^+ & \alpha_1^- & \alpha_0^+  \\
\beta_1^+ & \beta_1^- & \beta_0^+\\
\gamma_1^+ & \gamma_1^- & \gamma_0^+
\end{pmatrix}\cdot
\begin{pmatrix}
\mathcal{E}_\xi f_1^+ \\
\mathcal{E}_\xi f_1^- \\
\mathcal{E}_\xi f_0^+
\end{pmatrix},
\]
it then follows
\[
\begin{pmatrix}
\alpha_1^+ & \alpha_1^- & \alpha_0^+  \\
\beta_1^+ & \beta_1^- & \beta_0^+\\
\gamma_1^+ & \gamma_1^- & \gamma_0^+
\end{pmatrix}
\cdot
\left(\begin{array}{lll}
\left(\mathcal{E}_\xi f_1^{+}, f_1^{+}\right) & \left(\mathcal{E}_\xi f_1^{+}, f_1^{-}\right) & \left(\mathcal{E}_\xi f_1^{+}, f_0^{+}\right) \\
\left(\mathcal{E}_\xi f_1^{-}, f_1^{+}\right) & \left(\mathcal{E}_\xi f_1^{-}, f_1^{-}\right) & \left(\mathcal{E}_\xi f_1^{-}, f_0^{+}\right) \\
\left(\mathcal{E}_\xi f_0^{+}, f_1^{+}\right) & \left(\mathcal{E}_\xi f_0^{+}, f_1^{-}\right) & \left(\mathcal{E}_\xi f_0^{+}, f_0^{+}\right)
\end{array}\right)
=E_{3\times 3}.
\]
Recalling the Definition \ref{def4.1}, we compute that
\[
\begin{pmatrix}
\alpha_1^+ & \alpha_1^- & \alpha_0^+  \\
\beta_1^+ & \beta_1^- & \beta_0^+\\
\gamma_1^+ & \gamma_1^- & \gamma_0^+
\end{pmatrix}=\mathrm{J}_\xi^{\dagger}
\]
with
\begin{equation}\label{4.10}
\mathrm{J}_\xi:=\begin{pmatrix}
\frac{i \xi  \left(k^2 \left(\xi ^2-1\right)+1\right)}{k^4 \left(\xi ^2-1\right)^2+2 k^2 \left(\xi ^2+1\right)+1} & \frac{1-k^2 \left(\xi ^2-1\right)}{k^4 \left(\xi ^2-1\right)^2+2 k^2 \left(\xi ^2+1\right)+1} & 0 \\
\frac{k^2 \left(\xi ^2-1\right)-1}{k^4 \left(\xi ^2-1\right)^2+2 k^2 \left(\xi ^2+1\right)+1} & \frac{i \xi  \left(k^2 \left(\xi ^2-1\right)+1\right)}{k^4 \left(\xi ^2-1\right)^2+2 k^2 \left(\xi ^2+1\right)+1} & 0 \\
0 & 0 & \frac{i \xi }{k^2 \xi ^2+1}
\end{pmatrix},
\end{equation}
where $\mathrm{J}_\xi^{\dagger}=-\mathrm{J}_\xi$. Hence, we get from the symmetry of $\mathcal{A}_{\xi, a}$ and the relation between direct and dual bases that
\[\begin{aligned}
\mathrm{A}_{\xi, a}=\left(\left(\mathrm{J}_\xi^{\dagger} \mathcal{E}_\xi f_{m}^{\sigma}, \mathcal{A}_{\xi, a} f_{m^{\prime}}^{\sigma^{\prime}}\right)\right)_{m, m^{\prime}=0,1}^{\sigma, \sigma^{\prime}= \pm}&=
-\mathrm{J}_\xi \left(\left( \mathcal{E}_\xi f_{m}^{\sigma}, \mathcal{J}_\xi\mathcal{L}_{\xi, a} f_{m^{\prime}}^{\sigma^{\prime}}\right)\right)_{m, m^{\prime}=0,1}^{\sigma, \sigma^{\prime}= \pm}\\
&=
\mathrm{J}_\xi \left(\left(\mathcal{L}_{\xi, a}f_{m}^{\sigma},  f_{m^{\prime}}^{\sigma^{\prime}}\right)\right)_{m, m^{\prime}=0,1}^{\sigma, \sigma^{\prime}= \pm}.
\end{aligned}
\]
Moreover, as $(f, g)=\overline{(\mathcal{S} f, \mathcal{S} g)}$ for any $f, g \in L^2(\mathbb{T})$. In view of the properties of reversible basis and the reversibility preserving of $\mathcal{L}_{\xi, a}$
$$
\left(\mathcal{L}_{\xi, a} f_m^\sigma, f_{m^{\prime}}^{\sigma^{\prime}}\right)=\overline{\left(\mathcal{S} \mathcal{L}_{\xi, a} f_m^\sigma, \mathcal{S} f_{m^{\prime}}^{\sigma^{\prime}}\right)}=\overline{\left(\mathcal{L}_{\xi, a} \mathcal{S} f_m^\sigma, \mathcal{S} f_{m^{\prime}}^{\sigma^{\prime}}\right)}=\sigma \sigma^{\prime} \overline{\left(\mathcal{L}_{\xi, a} f_m^\sigma, f_{m^{\prime}}^{\sigma^{\prime}}\right)}.
$$
Summing up the above discussion, we deduce the following lemma regarding the expression of the operator $\left.\mathcal{A}_{\xi, a}\right|_{\mathcal{V}_{\xi, a}}$ on any $\xi$-symplectic and reversible basis.
\begin{lemma}[Matrix representation of $\mathcal{A}_{\xi, a}$ on $\mathcal{V}_{\xi, a}$]\label{lem4.2}Let $\left\{f_1^{+}, f_1^{-}, f_0^+\right\}$ a $\xi$-symplectic basis of $\mathcal{V}_{\xi, a}$ according to Definition \ref{def4.1}. The $3 \times 3$ matrix that represents the Hamiltonian and reversible operator $\mathcal{A}_{\xi, a}=\mathcal{J}_{\xi} \mathcal{L}_{\xi, a}$ with respect to any $\xi$-symplectic and reversible basis $\mathrm{F}$ of $\mathcal{V}_{\xi, a}$ is
\begin{equation}\label{4.12}
\mathrm{A}_{\xi, a}=\mathrm{J}_\xi \mathrm{L}_{\xi, a},
\end{equation}
where $\mathrm{J}_\xi$ is given in \eqref{4.10} for all $\xi\in[0,\frac 1 2]$ and
\begin{equation}
\mathrm{L}_{\xi, a}:=\left(\begin{array}{ccc}
\left(\mathcal{L}_{\xi, a} f_1^{+}, f_1^{+}\right) & \left(\mathcal{L}_{\xi, a} f_1^{-}, f_1^{+}\right) & \left(\mathcal{L}_{\xi, a} f_0^+, f_1^{+}\right) \\
\left(\mathcal{L}_{\xi, a} f_1^{+}, f_1^{-}\right) & \left(\mathcal{L}_{\xi, a} f_1^{-}, f_1^{-}\right) & \left(\mathcal{L}_{\xi, a} f_0^+, f_1^{-}\right) \\
\left(\mathcal{L}_{\xi, a} f_1^{+}, f_0\right) & \left(\mathcal{L}_{\xi, a} f_1^{-}, f_0\right) & \left(\mathcal{L}_{\xi, a} f_0^+, f_0^+\right)
\end{array}\right).
\end{equation}
Moreover, the entries of the matrix $\mathrm{L}_{\xi, a}$ are alternatively real or purely imaginary:
$$
\left(\mathcal{L}_{\xi, a} f_m^\sigma, f_{m^{\prime}}^\sigma\right) \in \mathbb{R}, \quad\left(\mathcal{L}_{\xi, a} f_m^\sigma, f_{m^{\prime}}^{-\sigma}\right) \in i \mathbb{R}.
$$
\end{lemma}
According to \cite{2022INVENT}, the matrix $\mathrm{J}_\xi$ is termed \emph{reversible}, $\mathrm{L}_{\xi,a}$ is termed \emph{reversibility preserving}, and $\mathrm{A}_{\xi, a}$ is termed \emph{Hamiltonian reversible}. In the following sections, we consider $|\xi|\ll 1$, then $J_\xi$ in \eqref{4.10} can be rewrite as
\begin{equation}\label{5.24}
\mathrm{J}_\xi=\begin{pmatrix}
\frac{i \left(1-k^2\right) \xi }{\left(k^2+1\right)^2}+r_1\left(\xi ^3\right) & \frac{1}{k^2+1}+\frac{k^2\left(k^2-3\right) \xi ^2}{\left(k^2+1\right)^3}+r_2\left(\xi^4\right) & 0 \\
-\left(\frac{1}{k^2+1}+\frac{k^2\left(k^2-3\right) \xi ^2}{\left(k^2+1\right)^3}+r_2\left(\xi^4\right)\right) & \frac{i \left(1-k^2\right) \xi }{\left(k^2+1\right)^2}+r_1\left(\xi ^3\right) & 0 \\
0 & 0 & i \xi +r_7\left(\xi ^3\right)
\end{pmatrix}.
\end{equation}


\section{Matrix representation of $\mathcal{A}_{\xi, a}$ on  $\mathcal{V}_{\xi, a}$}\label{sec Matrix}
This section presents the construction of a symplectic and reversible basis of $\mathcal{V}_{\xi, a}$ through application of the transformation operators $U_{\xi, a}$ developed earlier. Further more, we compute in Proposition \ref{pro5.1} the $3 \times 3$ Hamiltonian and reversible matrix representing the action of $\mathcal{A}_{\xi, a}: \mathcal{V}_{\xi, a} \rightarrow \mathcal{V}_{\xi, a}$ on such basis.
Since $\mathcal{V}_{0,0}=\operatorname{span}\left\{f_1^{+}, f_1^{-}, f_0^+\right\}:=\operatorname{span}\left\{\cos z, \sin z, \frac 1 {\sqrt{2}}\right\}$, \eqref{4.5} implies
$$
\mathcal{V}_{\xi, a}=\operatorname{span}\left\{f_1^{+}(\xi, a), f_1^{-}(\xi, a), f^+_0(\xi, a)\right\},
$$
where
\begin{equation}\label{5.1}
f_m^\sigma(\xi, a):=U_{\xi, a} f_m^\sigma, \quad(m, \sigma) \in\{(1, \pm),(0,+)\}.
\end{equation}
We note that $f_m^\sigma(\xi, a)$ is analytic in $(\xi, a)$ near $(0,0)$ because $U_{\xi, a}$ is so. A similar discussion as in \cite[Lemma 3.4]{ma2024} shows $U_{\xi, a}$ are symplectic, i.e.
\[
U_{\xi, a}^* \mathcal{E}_\xi U_{\xi, a}=\mathcal{E}_\xi \quad \text{for} \;\xi\neq 0, \qquad U_{0, a}^* \mathcal{E}_0 U_{0, a}=\mathcal{E}_0\quad \text{in}\; L_0^2(\mathbb{T}),
\]
and are reversibility preserving. Moreover, $U_{0, a}\mathcal{J}_0=\mathcal{J}_0U_{0, a}^{-*}$.
It follows that
\[
\left(\mathcal{E}_\xi U_{\xi, a} f_m^\sigma, U_{\xi, a} f_{m^{\prime}}^{\sigma^{\prime}}\right)=(\mathcal{E}_\xi f_m^\sigma, f_{m^{\prime}}^{\sigma^{\prime}})\] and hence the Kato basis $\mathrm{F}=\{f_m^\sigma(\xi, a)\}_{m=0,1}^{\sigma= \pm}$ is $\xi$-symplectic. On the other hand, since $\mathcal{S}f_m^\sigma(\xi, a)=\mathcal{S}U_{\xi, a}f_m^\sigma=U_{\xi, a}\mathcal{S}f_m^\sigma=\sigma U_{\xi, a} f_m^\sigma=\sigma f_m^\sigma(\xi, a)$, the Kato basis $\mathrm{F}=\{f_m^\sigma(\xi, a)\}_{m=0,1}^{\sigma= \pm}$ is reversible.

We now establish a lemma concerning the parametric expansion of the vectors $f_m^\sigma(\xi, a)$ in terms of $\xi, a$.
\begin{lemma}[Expansion of the basis $\mathrm{F}$]\label{lem5.1} For small values of $\xi, a$ the vectors $f_m^\sigma(\xi, a)$ in \eqref{5.1} expand as
\begin{equation}\label{5.2}
\begin{aligned}
& f_1^{+}(\xi, a)=\cos z+ 2 a A_2 \cos (2 z)-2 i \xi a A_2 \sin (2 z)+a^2 \text {even}_0(z)+\mathcal{O}\left(\xi^2 a, \xi a^2, a^3\right), \\
& f_1^{-}(\xi, a)=\sin z+ 2 a A_2 \sin (2 z)+2 i \xi a A_2 \cos (2 z)+a^2 \text {odd}(z)+\mathcal{O}\left(\xi^2 a, \xi a^2, a^3\right), \\
& f_0^{+}(\xi, a)=\frac 1 {\sqrt 2}+a^2 even_1
+\mathcal{O}\left(\xi^2 a, \xi a^2, a^3\right),
\end{aligned}
\end{equation}
where $even_0(x)$ is a real valued, even function with zero average and
\begin{equation*}
even_1:=\frac{\left(1+k^2\right)^{4/3}}{8\sqrt 2 b^2 k^2}-\frac{\left(1+k^2\right)^{4/3}}{4\sqrt 2 b^2 k^2}\cos(2z).
\end{equation*}
Finally, doting $\partial_\xi$ with a dot, one has the expansions
\begin{equation}\label{5.3}
\begin{aligned}
&\dot{f}_1^{+}(0, a)=-i a 2A_2 \sin (2 z)+\mathcal{O}\left(a^2\right),\quad \dot{f}_1^{-}(0, a)=i a 2A_2 \cos (2 z)+\mathcal{O}\left(a^2\right),\\
&\dot{f}_0^+(0, a)=\mathcal{O}\left(a^2\right), \quad \ddot{f}_m^\sigma(0, a)=\mathcal{O}(a).
\end{aligned}
\end{equation}
\end{lemma}
The long computations are provided in Appendix \ref{app B}. We now state the main result of this section.

\begin{proposition}\label{pro5.1} The action of the Hamiltonian and reversible operator $\mathcal{A}_{\xi, a}$ in the symplectic and reversible basis $\left\{f_1^{+}(\xi, a), f_1^{-}(\xi, a), f_0^{+}(\xi, a)\right\}$ of $\mathcal{V}_{\xi, a}$, defined in \eqref{5.1}, is represented by the $3 \times 3$ Hamiltonian and reversible matrix
\begin{equation}\label{5.4}
\mathrm{A}_{\xi, a}=\mathrm{J}_\xi \mathrm{L}_{\xi, a},
\end{equation}
where $\mathrm{J}_\xi$ is in \eqref{5.24} and $\mathrm{L}_{\xi, a}$ is the self-adjoint and reversibility preserving $3 \times 3$-matrix
\begin{equation}\label{5.5}
\mathrm{L}_{\xi, a}=\left(\begin{array}{c|c}
E & \mathrm{f} \\
\hline \mathrm{f}^{\dagger} & g
\end{array}\right).
\end{equation}
Here $E$ is the $2 \times 2$ selfadjoint matrix
\begin{equation}\label{5.6}
E=\left(\begin{array}{cc}
-\mathrm{e}_{11}a^2\left(1+r_1^{\prime}(a)\right)+\mathrm{e}_{22} \xi^2\left(1+r_1^{\prime \prime}(\xi, a)\right) & -i \xi\left(2\mathrm{e}_{22}+r_2\left(a^2, \xi a, \xi^2\right)\right) \\
i \xi\left(2\mathrm{e}_{22}+r_2\left(a^2, \xi a, \xi^2\right)\right) & \mathrm{e}_{22} \xi^2\left(1+r_4(a, \xi)\right)
\end{array}\right),
\end{equation}
where the vector $\mathrm{f}$ and the number $g$ are given by
\begin{equation}\label{5.7}
\mathrm{f}=\binom{-\sqrt{2} w_0 (k^2+3) a\left(1+r_3\left(a^2, \xi^2\right)\right)}{-i \xi a \sqrt{2} k^2 w_0\left(1+r_5(a, \xi)\right)}, \quad g=\mathrm{e}_{33}+\mathrm{g}_{33} a^2+\mathrm{e}_{22} \xi^2+r_6\left(a^3, \xi^2 a, \xi^3\right),
\end{equation}
with
\begin{equation}\label{**}
\begin{aligned}
\mathrm{e}_{11}&:=\frac{\left(k^2+1\right) \left(2 k^4+9 k^2+3\right)}{2 k^2}, \quad
\mathrm{e}_{22}:=k^2 \left(c_0-w_0^2\right)=\frac {2k^2}{1+k^2}w_0^2,\\
\mathrm{e}_{33}&:=c_0-3w_0^2=-\mathrm{e}_{22},\quad
\mathrm{g}_{33}:=\frac{3}{4} \left(\frac{1}{k^2}+1\right).
\end{aligned}
\end{equation}
\end{proposition}
The remainder of this section is dedicated to establishing the proof of Proposition \ref{pro5.1}. Since $f_m^\sigma(\xi, a)$ is analytic in $(\xi, a)$ near $(0,0)$, then the matrix $\mathrm{L}_{\xi, a}:=\left(\mathcal{L}_{\xi, a} f_m^\sigma(\xi, a), f_{m^{\prime}}^{\sigma^{\prime}}(\xi, a)\right)$ can be expanded in $\xi$, up to second order:
\begin{equation}\label{5.8}
\mathrm{L}_{\xi, a}=\mathrm{L}_{0, a}+\xi \dot{\mathrm{L}}_{0, a}+\frac{\xi^2}{2} \ddot{\mathrm{L}}_{0, a}+\xi^2\mathcal{O}\left(\xi,a\right).
\end{equation}
We proceed to analyze each component term separately.

\noindent \textbf{Expansion of $\mathrm{L}_{0, a}$.} We start with the matrix $\mathrm{L}_{0, a}=\left(\left(\mathcal{L}_{0, a} f_m^\sigma(0, a), f_{m^{\prime}}^{\sigma^{\prime}}(0, a)\right)\right)_{m, m^{\prime}=0,1}^{\sigma, \sigma^{\prime}= \pm}$.

\begin{lemma}[Expansion of the matrix $\mathrm{L}_{0, a}$]\label{lem5.2} The $3 \times 3$ real and reversibility preserving matrix $\mathrm{L}_{0, a}$ expands as
\begin{equation}\label{5.9}
\mathrm{L}_{0, a}=\left(\begin{array}{ccc}
-a^2\frac{\left(k^2+1\right) \left(2 k^4+9 k^2+3\right)}{2 k^2}+r_1\left(a^3\right) & 0 & -\sqrt{2} w_0 (k^2+3) a +r_2\left(a^3\right) \\
0 & 0 & 0 \\
-\sqrt{2} w_0 (k^2+3) a +r_2\left(a^3\right) & 0 & c_0-3 w_0^2+\frac{3}{4} \left(\frac{1}{k^2}+1\right)a^2 +r_3\left(a^3\right)
\end{array}\right).
\end{equation}
\end{lemma}
\begin{proof} As a consequence of $\mathcal{L}_{0, a}$ being a real operator and the vectors $f_m^\sigma(0, a)$ being real, the matrix $\mathrm{L}_{0, a}$ is. Additionally, Lemma \ref{lem4.2} establishes that the entries $\left(\mathcal{L}_{0, a} f_m^\sigma(0, a), f_{m^{\prime}}^{-\sigma}(0, a)\right)$ are purely imaginary, yielding
\begin{equation}\label{5.10}
\mathrm{B}_{0, a}=\left(\begin{array}{ccc}
E_{11}(a) & 0 & E_{13}(a) \\
0 & E_{22}(a) & 0 \\
E_{31}(a) & 0 & E_{33}(a)
\end{array}\right)
\end{equation}
for some real functions $E_{j k}(a)$.

\emph{Step 1: proof that $E_{22}(a) \equiv 0$ for $|a|$ sufficiently small.} As established in Section \ref{sec3.2}, the operator $\mathcal{A}_{0, a}$ possesses zero as an eigenvalue with algebraic multiplicity three and geometric multiplicity two. Consequently, the matrix $\mathrm{J}_0\mathrm{L}_{0, a}$-- representing the operator's action on the generalized eigenspace $\mathcal{V}_{0, a}$--exhibits two Jordan blocks: one of dimension 1 and the other of dimension 2. This structure implies that $\mathrm{J}_0 \mathrm{L}_{0, a}$ is nilpotent of index 2, i.e.
$$
0=\left(\mathrm{J}_0 \mathrm{L}_{0, a}\right)^2=\left(\begin{array}{ccc}
-E_{11}(a) E_{22}(a) & 0 & -E_{22}(a) E_{13}(a) \\
0 & -E_{11}(a) E_{22}(a) & 0 \\
0 & 0 & 0
\end{array}\right).
$$
As showed below that $E_{11}(a)\neq0, E_{13}(a)\neq0$ for $|a|$ sufficiently small, which gives $E_{22}(a) \equiv 0.$

\emph{Step 2: expansion of $E_{11}(a), E_{33}(a)$ and $E_{13}(a)$.} Note that
\begin{equation}\label{5.11}
\mathcal{L}_{0, a}=
-k^2\partial_z\left(\left(c-w^2+k^2w_{z}^2\right)\partial_z\right)+c-3 w^2+k^2w_{z}^2+2 k^2 w w_{z z}.
\end{equation}
This, coupled with the expression of $w$, led to arduous calculations, so we take advantage of Mathematica to carry out that
\begin{equation}\label{5.12}
\begin{aligned}
E_{11}(a)&=\left(\mathcal{L}_{0, a} f_1^+(0, a), f_{1}^{+}(0, a)\right)=-a^2\frac{\left(k^2+1\right) \left(2 k^4+9 k^2+3\right)}{2 k^2}+\mathcal{O}\left(a^3\right), \\
E_{13}(a)&=E_{31}(a)=\left(\mathcal{L}_{0, a} f_1^+(0, a), f_{0}^{+}(0, a)\right)=-\sqrt{2} w_0 (k^2+3) a +\mathcal{O}\left(a^3\right),\\
E_{33}(a)&=\left(\mathcal{L}_{0, a} f_0^+(0, a), f_{0}^{+}(0, a)\right)=c_0-3 w_0^2+ \frac{3}{4}\left(1+\frac{1}{k^2}\right)a^2+\mathcal{O}\left(a^3\right).
\end{aligned}
\end{equation}
Then the expansion \eqref{5.9} follows from step 1 and \eqref{5.12}.
\end{proof}
\noindent\textbf{Expansion of $\dot{\mathrm{L}}_{0, a}$.} The following analysis concerns the expansion of the matrix $\dot{\mathrm{L}}_{0, a}$ in \eqref{5.8}. Note that
\begin{equation}\label{5.13}
\dot{\mathrm{L}}_{0, a}=\left(\left(\dot{\mathcal{L}}_{0, a} f_m^\sigma(0, a), f_{m^{\prime}}^{\sigma^{\prime}}(0, a)\right)+2 \operatorname{Sym}\left(\mathcal{L}_{0, a} f_m^\sigma(0, a), \dot{f}_{m^{\prime}}^{\sigma^{\prime}}(0, a)\right)\right)=\dot{\mathrm{L}}_{0, a}^{[1]}+\dot{\mathrm{L}}_{0, a}^{[2]}+\dot{\mathrm{L}}_{0, a}^{[2]*},
\end{equation}
where $X^*$ denotes the transpose conjugate matrix to $X$.
\begin{lemma}[Expansion of the matrix $\dot{\mathrm{L}}_{0, a}$]\label{lem5.3} The $3 \times 3$ selfadjoint, purely imaginary and reversibility preserving matrix $\dot{\mathrm{L}}_{0, a}$ in \eqref{5.13} expands as
\begin{equation}\label{5.14}
\dot{\mathrm{L}}_{0, a}=\left(\begin{array}{ccc}
0 & i\left(2 k^2 \left(w_0^2-c_0\right)+r_4\left(a^2\right)\right) & 0 \\
-i\left(2 k^2 \left(w_0^2-c_0\right)+r_4\left(a^2\right)\right) & 0 & -i \left(a \sqrt{2} k^2 w_0+ r_5\left(a^2\right) \right)\\
0 & i \left(a \sqrt{2} k^2 w_0+ r_5\left(a^2\right) \right) & 0
\end{array}\right).
\end{equation}
\end{lemma}
\begin{proof}We consider the matrices $\dot{\mathrm{L}}_{0, a}^{[1]}$ and $\dot{\mathrm{L}}_{0, a}^{[2]}$ in \eqref{5.13} separately .

\noindent\emph{Expansion of $\dot{L}_{0, a}^{[1]}$.} Recall $\dot{\mathrm{L}}_{0, a}^{[1]}=\left(\dot{\mathcal{L}}_{0, a} f_m^\sigma(0, a), f_{m^{\prime}}^{\sigma^{\prime}}(0, a)\right)_{m, m^{\prime}=0,1}^{\sigma, \sigma^{\prime}= \pm}$. As a consequence of the reversibility preservation of
\[
\dot{\mathcal{L}}_{0, a}=-i k^2\left(\left(c-w^2+k^2w_{z}^2\right)\partial_z+\partial_z\left(c-w^2+k^2w_{z}^2\right)\right),
\]
the entries of the matrix are alternatively real and purely imaginary. Given that $\dot{\mathcal{L}}_{0, a}$ is also purely imaginary, every entry must also be purely imaginary. This implies that for sufficiently small $|a|$, the entries of the type $\left(\dot{\mathcal{L}}_{0, a} f_m^\sigma(0, a), f_{m^{\prime}}^\sigma(0, a)\right)$ vanish identically:
\begin{equation}\label{5.15}
\dot{\mathrm{L}}_{0, a}^{[1]}=\left(\begin{array}{ccc}
0 & \dot{E}_{12}^{[1]}(a) & 0 \\
\dot{E}_{21}^{[1]}(a) & 0 & \dot{E}_{23}^{[1]}(a) \\
0 & \dot{E}_{32}^{[1]}(a) & 0
\end{array}\right).
\end{equation}
Taking also advantage of Mathematica gives the expansion of the remaining elements as
\begin{equation}\label{5.16}
\begin{aligned}
\dot{E}_{12}^{[1]}(a)&=\left(\dot{\mathcal{L}}_{0, a} f_1^{-}(0, a), f_1^{+}(0, a)\right)=-\dot{E}_{21}^{[1]}(a)=2i k^2 \left( w_0^2-c_0\right)+\mathcal{O}\left(a^2\right), \\
\dot{E}_{23}^{[1]}(a)&=\left(\dot{\mathcal{L}}_{0, a} f_0^{+}(0, a), f_1^{-}(0, a)\right)
=-\dot{E}_{32}^{[1]}(a)=-i a \sqrt{2} k^2 w_0+\mathcal{O}\left(a^2\right).
\end{aligned}
\end{equation}

\noindent\emph{Expansion of $\dot{L}_{0, a}^{[2]}$.} Recall $\dot{L}_{0, a}^{[2]}=\left(\left(\mathcal{L}_{0, a} f_m^\sigma(0, a), \dot{f}_{m^{\prime}}^{\sigma^{\prime}}(0, a)\right)\right)$. The reversibility of basis $\mathrm{F}$ ensures that $\dot{f}_m^\sigma(0, a)$ maintains the parity characteristics outlined in \eqref{4.8}. The purely imaginary character of $\dot{U}_{0, a}$ in \eqref{*} combined with the fact that $f_m^\sigma$ is real renders $\dot{f}_m^\sigma(0, a)=\dot{U}_{0, a} f_m^\sigma$ purely imaginary. While $f_m^\sigma(0, a)$ are real, yielding
\begin{equation}\label{5.17}
f_m^{+}(0, a)=\operatorname{even}(z), \quad f_m^{-}(0, a)=\operatorname{odd}(z), \quad \dot{f}_m^{+}(0, z)=i \operatorname{odd}(z), \quad \dot{f}_m^{-}(0, a)=i \operatorname{even}(z).
\end{equation}
Note that $\mathcal{L}_{0, a}$ given in \eqref{5.11} is parity preserving, so the matrix $\dot{\mathrm{L}}_{0, a}^{[2]}+\dot{\mathrm{L}}_{0, a}^{[2] *}$ takes the form
\begin{equation}\label{5.18}
\dot{\mathrm{L}}_{0, a}^{[2]}+\dot{\mathrm{L}}_{0, a}^{[2] *}=\left(\begin{array}{ccc}
0 & \dot{E}_{12}^{[2]}(a) & 0 \\
-\dot{E}_{12}^{[2]}(a) & 0 & \dot{E}_{23}^{[2]}(a) \\
0 & -\dot{E}_{23}^{[2]}(a) & 0
\end{array}\right)
\end{equation}
with some coefficients that we now compute.
\begin{equation}\label{5.19}
\begin{aligned}
\dot{E}_{12}^{[2]}(a)&=\left(\mathcal{L}_{0, a} f_1^{-}(0, a), \dot{f}_1^{+}(0, a)\right)+\overline{\left(\mathcal{L}_{0, a} f_1^{+}(0, a), \dot{f}_1^{-}(0, a)\right)}=\mathcal{O}\left(a^2\right),\\
\dot{E}_{23}^{[2]}(a)&=\left(\mathcal{L}_{0, a} f_0^+(0, a), \dot{f}_1^{-}(0, a)\right)+\overline{\left(\mathcal{L}_{0, a} f_1^{-}(0, a), \dot{f}_0^{+}(0, a)\right)}
=\mathcal{O}\left(a^2\right).
\end{aligned}
\end{equation}
In conclusion the equations \eqref{5.15}-\eqref{5.19} give the expansion \eqref{5.14}.
\end{proof}
\noindent\textbf{Expansion of $\ddot{\mathrm{L}}_{0, a}$.} As the final step, we compute the expansion for the matrix $\ddot{\mathrm{L}}_{0, a}$ in \eqref{5.8}, which admits the following presentation
\begin{equation}\label{5.20}
\begin{aligned}
\ddot{\mathrm{L}}_{0, a}=&\left.\left(\left(\ddot{\mathcal{L}}_{0, a} f_m^\sigma, f_{k^{\prime}}^{\sigma^{\prime}}\right)+\left(\mathcal{L}_{0, a} \ddot{f}_m^\sigma, f_{k^{\prime}}^{\sigma^{\prime}}\right)+\left(\mathcal{L}_{0, a} f_m^\sigma, \ddot{f}_{m^{\prime}}^{\sigma^{\prime}}\right)\right.\right.\\
&\left.\left.+2\left(\dot{\mathcal{L}}_{0, a} \dot{f}_m^\sigma, f_{k^{\prime}}^{\sigma^{\prime}}\right)+2\left(\dot{\mathcal{L}}_{0, a} f_m^\sigma, \dot{f}_{m^{\prime}}^{\sigma^{\prime}}\right)+2\left(\mathcal{L}_{0, a} \dot{f}_m^\sigma, \dot{f}_{m^{\prime}}^{\sigma^{\prime}}\right)\right)\right|_{\xi=0}.
\end{aligned}
\end{equation}
\begin{lemma}[Expansion of the matrix $\ddot{\mathrm{L}}_{0, a}$]\label{lem5.4}
The $3 \times 3$ selfadjoint, real and reversibility preserving matrix $\ddot{\mathrm{L}}_{0, a}$ in \eqref{5.20} expands as
\begin{equation}\label{5.21}
\ddot{\mathrm{L}}_{0, a}=\left(\begin{array}{ccc}
2 k^2 \left(c_0-w_0^2\right) & 0 & 0 \\
0 & 2 k^2 \left(c_0-w_0^2\right) & 0 \\
0 & 0 & 2 k^2 \left(c_0-w_0^2\right)
\end{array}\right)+\mathcal{O}(a).
\end{equation}
\end{lemma}
\begin{proof}
By Lemma \ref{lem5.1}, $\dot{f}_m^\sigma(0, a), \ddot{f}_m^\sigma(0, a)=\mathcal{O}(a)$. Hence, all the terms in \eqref{5.20} are $\mathcal{O}(a)$, except $\left(\ddot{\mathcal{L}}_{0, a} f_m^\sigma, f_{m^{\prime}}^{\sigma^{\prime}}\right)$. Since
\[
\ddot{\mathcal{L}}_{0, a}=2k^2\left(c-w^2+k^2w_{z}^2\right),
\]
it follows that
\begin{equation}\label{5.22}
\begin{aligned}
\left(\ddot{\mathcal{L}}_{0, a} f_1^{+}(0, a), f_1^{+}(0, a)\right)&=2 k^2 \left(c_0-w_0^2\right)+\mathcal{O}\left(a\right),\\
\left(\ddot{\mathcal{L}}_{0, a} f_1^{-}(0, a), f_1^{-}(0, a)\right)&=2 k^2 \left(c_0-w_0^2\right)+\mathcal{O}\left(a\right),\\
\left(\ddot{\mathcal{L}}_{0, a} f_0^{+}(0, a), f_0^{+}(0, a)\right)&=2 k^2 \left(c_0-w_0^2\right)+\mathcal{O}\left(a\right),\\
\left(\ddot{\mathcal{L}}_{0, a} f_1^{+}(0, a), f_1^{-}(0, a)\right)&=\left(\ddot{\mathcal{L}}_{0, a} f_1^{-}(0, a), f_1^{+}(0, a)\right)\\
&=\left(\ddot{\mathcal{L}}_{0, a} f_1^{-}(0, a), f_0^{+}(0, a)\right)=\left(\ddot{\mathcal{L}}_{0, a} f_0^{+}(0, a), f_1^{-}(0, a)\right)=0,\\
\left(\ddot{\mathcal{L}}_{0, a} f_1^{+}(0, a), f_0^{+}(0, a)\right)&=\left(\ddot{\mathcal{L}}_{0, a} f_0^{+}(0, a), f_1^{+}(0, a)\right)=\mathcal{O}\left(a\right).
\end{aligned}
\end{equation}
\end{proof}
Now, we are in the position to prove Proposition \ref{pro5.1}.
\begin{proof}[Proof of Proposition \ref{pro5.1}]Following from \eqref{5.8} and Lemmata \ref{lem5.2}-\ref{lem5.4}, the entries of the matrix $\mathrm{L}_{\xi, a}$ have the expressions
\begin{equation*}
\begin{aligned}
& \left(\mathrm{L}_{\xi, a}\right)_{1,1}=-a^2\frac{\left(k^2+1\right) \left(2 k^4+9 k^2+3\right)}{2 k^2}+\mathcal{O}\left(a^3\right)+\frac{1}{2} \xi^2\left(2 k^2 \left(c_0-w_0^2\right)+\mathcal{O}\left(a\right)\right)+\xi^2\mathcal{O}\left(\xi,a\right),\\
& \left(\mathrm{L}_{\xi, a}\right)_{1,2}=i\left(\xi 2 k^2 \left(w_0^2 -c_0\right)+\xi\mathcal{O}\left(a^2\right)+\xi^2 \mathcal{O}(a)\right)+\xi^2\mathcal{O}(\xi,a), \\
& \left(\mathrm{L}_{\xi, a}\right)_{2,2}=\frac{1}{2} \left(2 k^2 \left(c_0-w_0^2\right)\right) \xi^2+\xi^2 \mathcal{O}(a)+\xi^2 \mathcal{O}\left(\xi,a\right), \\
& \left(\mathrm{L}_{\xi, a}\right)_{3,3}=c_0-3 w_0^2+\frac{3}{4} \left(\frac{1}{k^2}+1\right)a^2 +\mathcal{O}_3\left(a^3\right)+\frac{1}{2} \xi^2\left(2 k^2 \left(c_0-w_0^2\right)+\mathcal{O}\left(a\right)\right)+\xi^2 \mathcal{O}\left(\xi,a\right),\\
&\left(\mathrm{L}_{\xi, a}\right)_{1,3}=-\sqrt{2} w_0 (k^2+3) a\left(1+\mathcal{O}\left(a^2, \xi^2\right)\right), \quad\left(\mathrm{L}_{\xi, a}\right)_{2,3}=-i\xi a\sqrt{2} k^2 w_0\left(1+ \mathcal{O}(\xi, a)\right).
\end{aligned}
\end{equation*}
This completes the proof.
\end{proof}

\section{Block-decoupling}\label{sec Block}
In this section, we focus on transforming the matrix $\mathrm{A}_{\xi, a}$ in \eqref{5.5} into block-diagonal form following very closely similar computations carried out in \cite{2022INVENT,ma2024}. The approach begins with a singular scaling transformation that is non-symplectic according to Definition \ref{def4.1}, necessitating a detailed computation of how this transformation modifies the Poisson tensor $\mathrm{J}_\xi$.
\begin{lemma}\label{lem6.1} The conjugation of the Hamiltonian and reversible matrix $\mathrm{A}_{\xi, a}$ with the reversibility preserving matrix
\begin{equation}\label{6.1}
Y:=\left(\begin{array}{c|c}
Q & 0 \\
\hline 0^{\dagger} & \sqrt{\xi}
\end{array}\right), \quad Q=\left(\begin{array}{cc}
\sqrt{\xi} & 0 \\
0 & \frac{1}{\sqrt{\xi}}
\end{array}\right), \quad \xi>0,
\end{equation}
yields the Hamiltonian and reversible matrix
\begin{equation}\label{6.2}
\mathrm{A}_{\xi, a}^{(1)}:=Y^{-1} \mathrm{A}_{\xi, a} Y=\xi \mathrm{J}_\xi^{(1)} \mathrm{L}_{\xi, a}^{(1)},
\end{equation}
where $\mathrm{J}_\xi^{(1)}$ is the skew-adjoint and reversible matrix
\begin{equation}\label{6.3}
\begin{aligned}
&\mathrm{J}_\xi^{(1)}:=Y^{-1} \mathrm{J}_\xi Y^{-*}=\l\left(\begin{array}{c|c}
\hat{\mathrm{J}}_\xi & 0 \\
\hline 0^{\dagger} & i +r_7\left(\xi ^2\right)
\end{array}\right), \quad \text{and}\\
&\hat{\mathrm{J}}_\xi=\left(\begin{array}{cc}
\frac{i \left(1-k^2\right)}{\left(k^2+1\right)^2}+r_1\left(\xi ^2\right) & \frac{1}{k^2+1}+\frac{\left(k^4-3 k^2\right) \xi ^2}{\left(k^2+1\right)^3}+r_2\left(\xi^3\right) \\
-\left(\frac{1}{k^2+1}+\frac{\left(k^4-3 k^2\right) \xi ^2}{\left(k^2+1\right)^3}+r_2\left(\xi^3\right)\right) & \frac{i \left(1-k^2\right)}{\left(k^2+1\right)^2}\xi^2+r_1\left(\xi ^4\right)
\end{array}\right),
\end{aligned}
\end{equation}
$\mathrm{L}_{\xi, a}^{(1)}$ is the selfadjoint and reversibility preserving matrix
\begin{equation}\label{6.4}
\mathrm{L}_{\xi, a}^{(1)}:=Y^* \mathrm{L}_{\xi, a} Y=\left(\begin{array}{c|c}
E^{(1)} & \mathrm{f}^{(1)} \\
\hline \mathrm{f}^{(1) \dagger} & g^{(1)}
\end{array}\right)
\end{equation}
and the $2 \times 2$ symmetric and reversibility preserving matrix $E^{(1)}$, the vector $f^{(1)}$ and the number $g^{(1)}$ expand as
\begin{equation}\label{6.5}
\begin{aligned}
& E^{(1)}=\left(\begin{array}{cc}
-\mathrm{e}_{11}a^2\left(1+r_1^{\prime}(a)\right)+\mathrm{e}_{22} \xi^2\left(1+r_1^{\prime \prime}(\xi, a)\right) & -i\left(2\mathrm{e}_{22}+r_2\left(a^2, \xi a, \xi^2\right)\right)  \\
i \left(2\mathrm{e}_{22}+r_2\left(a^2, \xi a, \xi^2\right)\right)  & \mathrm{e}_{22}+r_4(a, \xi)
\end{array}\right), \\
& \mathrm{f}^{(1)}=-\sqrt{2}a w_0 \binom{(k^2+3)+r_3\left(a^2, \xi^2\right)}{ik^2 + r_5(a, \xi)}, \quad g^{(1)}=g=\mathrm{e}_{33}+\mathrm{g}_{33} a^2+\mathrm{e}_{22} \xi^2+r_6\left(a^3, \xi^2 a, \xi^3\right).
\end{aligned}
\end{equation}
\end{lemma}
\begin{proof} The result is readily obtained through computation of matrix $Y^* \mathrm{L}_{\xi, a} Y=\left(\begin{array}{c|c}Q E Q & \sqrt{\xi} Q \mathrm{f} \\ \hline(\sqrt{\xi} Q \mathrm{f})^{\dagger} & \xi g\end{array}\right)$ followed by factorization of the common parameter $\xi$ from all matrix elements.
\end{proof}
For $\xi \neq 0$, the spectrum of $\mathrm{A}_{\xi, a}^{(1)}$ coincides with that of $\mathrm{A}_{\xi, a}$.
\subsection{Non-perturbative step of block decoupling}
We now proceed to reduce matrix $\mathrm{A}_{\xi, a}^{(1)}$ in \eqref{6.2} to block-diagonal structure. The main result of this subsection is the following.
\begin{lemma}\label{lem6.2} There exists a smooth vector $\mathrm{s}(\xi, a)$, with values in $\mathbb{C}^2$ of the form
\begin{equation}\label{6.6}
\mathrm{s}:=\mathrm{s}(\xi, a)=\frac{\sqrt{2} aw_0(k^2+3)}{\mathrm{e}_d^2}\binom{i\left(\mathrm{e}_d+r_1\left(a^2, \xi a, \xi^2\right)\right)}{\mathrm{e}_b-\frac {k^2}{k^2+3}\mathrm{e}_d+r_2(a, \xi)}
\end{equation}
with $\mathrm{e}_b, \mathrm{e}_d$ defined as
\begin{equation}\label{6.18}
\mathrm{e}_{d}=\frac{2 \mathrm{e}_{22}}{k^2+1}-\mathrm{e}_{33}=\frac{2 k^2 \left(k^2+3\right)}{\left(k^2+1\right)^{2}}w_0^2, \quad
\mathrm{e}_{b}=\frac{\mathrm{e}_{22} \left(3- k^2\right)}{\left(k^2+1\right)^2}=\frac{2 k^2 \left(3 -k^2\right)}{\left(k^2+1\right)^{3}}w_0^2,
\end{equation}
such that the following holds true. Conjugating the Hamiltonian and reversible matrix $\mathrm{A}_{\xi, a}^{(1)}$ in \eqref{6.2} with the symplectic and reversibility-preserving matrix
\begin{equation}\label{6.7}
\exp (S), \quad S:=\mathrm{J}_\xi^{(1)}\left(\begin{array}{c|c}
\mathbf{0} & \mathbf{s}(\xi, a) \\
\hline \mathbf{s}(\xi, a)^{\dagger} & 0
\end{array}\right)=\mathrm{J}_\xi^{(1)} X,
\end{equation}
we obtain the Hamiltonian and reversible matrix
\begin{equation}\label{6.8}
\mathrm{A}_{\xi, a}^{(2)}:=\xi \exp (S) \mathrm{A}_{\xi, a}^{(1)} \exp (-S)=\xi \mathrm{J}_\xi^{(1)} \mathrm{L}_{\xi, a}^{(2)},
\end{equation}
where $\mathrm{J}_\xi^{(1)}$ in \eqref{6.3} and
\begin{equation}\label{6.9}
\mathrm{L}_{\xi, a}^{(2)}:=\exp (S)^{-*} \mathrm{A}_{\xi, a}^{(1)} \exp (S)^{-1}=\left(\begin{array}{c|c}
E^{(2)} & \mathrm{f}^{(2)} \\
\hline \mathrm{f}^{(2) \dagger} & g^{(2)}
\end{array}\right)
\end{equation}
with the $2 \times 2$ symmetric and reversibility preserving matrix $E^{(2)}$, the vector $f^{(2)}$ and the number $g^{(2)}$ expanding as
\begin{equation}\label{6.10}
\begin{aligned}
E^{(2)}= & \left(\begin{array}{cc}
-\mathrm{e}_w a^2\left(1+r_1^{\prime}(a, \xi)\right)+\mathrm{e}_{22} \xi^2\left(1+r_1^{\prime \prime}(a, \xi)\right) & -i\left(2\mathrm{e}_{22}+r_2\left(a^2, \xi a, \xi^2\right)\right) \\
i\left(2\mathrm{e}_{22}+r_2\left(a^2, \xi a, \xi^2\right)\right) & \mathrm{e}_{22}+r_4(a, \xi)
\end{array}\right), \\
& \mathrm{f}^{(2)}=a^3\binom{r_3(1)}{\mathrm{i} r_5(1)}, \quad g^{(2)}=\mathrm{e}_{33}+r_{6}\left(a^2, \xi^2\right),
\end{aligned}
\end{equation}
and
\begin{equation}\label{6.11}
\mathrm{e}_w:=\mathrm{e}_{11}-\frac{2w_0^2(k^2+3)^2}{\mathrm{e}_d}=\frac{(k^2+1)(k^2-3)}{2k^2}.
\end{equation}
The functions $r_m$ are smooth functions in $(\xi, a)$.
\end{lemma}
We shall devote the subsequent analysis to establishing the validity of Lemma \ref{lem6.2}. Our analysis begins with seeking $\mathbf{s}=\binom{\mathrm{i} s_1(\xi, a)}{s_2(\xi, a)}$ with $s_j(\xi, a)$ real valued to guarantee that the matrix $X$  maintains  symmetric and reversible. By \eqref{6.3} and the arguments in \cite[Lemma 3.13]{2022INVENT} and \cite[Lemma 3.9]{ma2024}, the matrix $\exp (S)$ is $\mathrm{J}_\xi^{(1)}$-symplectic and reversibility preserving for any $\xi \neq 0$. We proceed to compute the Lie series expansion for the matrix $\mathrm{A}_{\xi, a}^{(2)}$ in \eqref{6.8}. The first step consists of separating $\mathrm{A}_{\xi, a}^{(1)}$ into block-diagonal and off-diagonal terms as follows:
\begin{equation}\label{6.12}
\mathrm{A}_{\xi, a}^{(1)}=\xi\left(D^{(1)}+R^{(1)}\right),
\end{equation}
with
\begin{equation}\label{6.13}
D^{(1)}:=\left(\begin{array}{c|c}
\hat{\mathrm{J}}_\xi E^{(1)} & 0 \\
\hline 0^{\dagger} & \left(i +r_7\left(\xi ^2\right)\right) g^{(1)}
\end{array}\right),\quad R^{(1)}:=\left(\begin{array}{c|c}
\mathbf{0} & \hat{\mathrm{J}}_\xi \mathrm{f}^{(1)} \\
\hline \left(i +r_7\left(\xi ^2\right)\right)\mathrm{f}^{(1) \dagger} & 0
\end{array}\right).
\end{equation}
The Lie expansion of $\mathrm{A}_{\xi, a}^{(2)}$ is
\begin{equation}\label{6.14}
\begin{gathered}
\mathrm{A}_{\xi, a}^{(2)}=\xi\left(D^{(1)}+R^{(1)}+\left[S, D^{(1)}\right]+\left[S, R^{(1)}\right]+\frac{1}{2}\left[S,\left[S, D^{(1)}\right]\right]++\frac{1}{2} \int_0^1(1-\tau)^2\right. \\
\left. \exp (\tau S) \operatorname{ad}_S^3\left(D^{(1)}\right) \exp (-\tau S) \mathrm{d} \tau+\int_0^1(1-\tau) \exp (\tau S) \operatorname{ad}_S^2\left(R^{(1)}\right) \exp (-\tau S) \mathrm{d} \tau\right),
\end{gathered}
\end{equation}
where $\operatorname{ad}_A(B):=[A, B]:=A B-B A$ denotes the commutator between linear operators $A, B$. We seek to determine the matrix $S$ that satisfies the equation $R^{(1)}+\left[S, D^{(1)}\right]=0$. Substituting the explicit expressions \eqref{6.6}, \eqref{6.12}, this equation becomes:
\begin{equation}\label{6.15}
\small
\left(\begin{array}{c|c}
\mathbf{0} & \left(\left(i +r_7\left(\xi ^2\right)\right) g^{(1)}-\widehat{\mathrm{J}}_\xi E^{(1)}\right) \widehat{\mathrm{J}}_\xi \mathrm{s}+\widehat{\mathrm{J}}_\xi \mathrm{f}^{(1)} \\
\hline \left(i +r_7\left(\xi ^2\right)\right)\left(\mathrm{s}^{\dagger}\left(\widehat{\mathrm{J}}_\xi E^{(1)}-\left(i +r_7\left(\xi ^2\right)\right) g^{(1)}\right)+\mathrm{f}^{(1) \dagger}\right) & 0
\end{array}\right)=0.
\end{equation}
Direct examination confirms the equivalence of these two equations, thus it suffices to solve
\begin{equation}\label{6.16}
\mathrm{s}^{\dagger}\left(\hat{\mathrm{J}}_\xi E^{(1)}-\left(i +r_7\left(\xi ^2\right)\right) g^{(1)}\right)=-\mathrm{f}^{(1) \dagger}.
\end{equation}
Recalling Lemma \ref{lem6.1}, the matrix \eqref{6.16} takes the form
\begin{equation}\label{6.17}
\begin{aligned}
&\hat{\mathrm{J}}_\xi E^{(1)}-\left(i +r_7\left(\xi ^2\right)\right)  g^{(1)}\\
&=\left(\begin{array}{cc}
i \mathrm{e}_{d}+i r_1\left(a^2, \xi a, \xi^2\right) &\mathrm{e}_{b} +r_2(a, \xi) \\
\frac{\mathrm{e}_{11}}{k^2+1} a^2\left(1+r_3(a)\right)-\mathrm{e}_{b}\xi^2\left(1+\tilde{r}_3(a, \xi)\right) & i \mathrm{e}_{d}+i r_4\left(a^2, \xi a, \xi^2\right)
\end{array}\right),
\end{aligned}
\end{equation}
with $\mathrm{e}_{d}, \mathrm{e}_{b}$ given in \eqref{6.18}. Its determinant is
\begin{equation}\label{6.19}
\operatorname{det}\left(\hat{\mathrm{J}}_\xi E^{(1)}-\left(i +r_7\left(\xi ^2\right)\right) g^{(1)}\right)=-\left(\mathrm{e}_d^2+r\left(a^2, \xi a, \xi^2\right)\right),
\end{equation}
which is not zero for sufficiently small $(\xi,a)$ since $\mathrm{e}_d\neq 0$. A direct computation then gives that the vector $\mathrm{s}(\xi,a)$ in \eqref{6.6} solves \eqref{6.16}.

Above computation shows that the matrix $S$ solves the homological equation $R^{(1)}+\left[S, D^{(1)}\right]=0$, the Lie expansion of $\mathrm{A}_{\xi, a}^{(2)}$ in \eqref{6.14} is then reduced to
\begin{equation}\label{6.20}
\mathrm{A}_{\xi, a}^{(2)}=\xi\left(D^{(1)}+\frac{1}{2}\left[S, R^{(1)}\right]+\frac{1}{2} \int_0^1\left(1-\tau^2\right) \exp (\tau S) \operatorname{ad}_S^2\left(R^{(1)}\right) \exp (-\tau S) \mathrm{d} \tau\right).
\end{equation}
Specifically, the block-diagonal correction $\frac{1}{2}\left[S, R^{(1)}\right]$ is the Hamiltonian and reversible matrix
\begin{equation}\label{6.21}
\begin{aligned}
&\frac{1}{2} \mathrm{J}_\xi^{(1)}\left(\begin{array}{c|c}
\left(i +r_7\left(\xi ^2\right)\right)\left(\mathrm{sf}^{(1) \dagger}-\mathrm{f}^{(1)} \mathrm{s}^{\dagger}\right) & 0 \\
\hline 0 & \left(i +r_7\left(\xi ^2\right)\right)\left(\mathrm{s}^{\dagger} \hat{J}_\xi \mathrm{f}^{(1)}-\mathrm{f}^{(1) \dagger} \hat{J}_\xi \mathrm{s}\right)
\end{array}\right)\\
&=: \mathrm{J}_\xi^{(1)}\left(\begin{array}{c|c}
\Delta E^{(1)} & 0 \\
\hline 0 & \Delta g^{(1)}
\end{array}\right),
\end{aligned}
\end{equation}
where the $2 \times 2$ selfadjoint and reversibility preserving matrix $\Delta E^{(1)}$ and the real number $\Delta g^{(1)}$ admit the following expansions
\begin{equation}\label{6.22}
\Delta E^{(1)}=\left(\begin{array}{cc}
\frac{2w_0^2(k^2+3)^2}{\mathrm{e}_d}a^2 +a^2r_1\left(a,\xi\right)& i r_2\left(a^2\right) \\
-i r_2(a^2) & r_3(a^2)
\end{array}\right), \quad \Delta g^{(1)}=r_4\left(a^2\right).
\end{equation}
We observe that although the other corrective terms are perturbative, the $(1, 1)$ entrance of the matrix $\Delta E^{(1)}$ is of the same order $a^2$ as that of $E^{(1)}$, thereby playing a fundamental role in the modulational instability analysis. The block-diagonal matrix $D^{(1)}+\frac{1}{2}\left[S, R^{(1)}\right]$ now takes the form
\begin{equation}\label{6.23}
\mathrm{J}_\xi^{(1)}\left(\begin{array}{c|c}
E^{(1)}+\Delta E^{(1)} & 0 \\
\hline 0 & g^{(1)}+\Delta g^{(1)}
\end{array}\right)=: \mathrm{J}_\xi^{(1)}\left(\begin{array}{c|c}
E^{(2)} & 0 \\
\hline 0 & g^{(2)}
\end{array}\right),
\end{equation}
with $E^{(2)}$ and $g^{(2)}$ given in \eqref{6.10}.

Finally, similar arguments as in \cite[Lemma 5.7, 5.8]{2022INVENT} show that the reminder of the Lie expansion in \eqref{6.20} are order of $a^3$ and for sufficiently small $|\xi|\ll 1,|a|\ll 1$, there exist a reversibility preserving, Hamiltonian matrix $S^{(2)}=S^{(2)}(\xi, a)$ such that
\begin{equation}\label{6.24}
\mathrm{A}_{\xi, a}^{(3)}:=\exp \left(S^{(2)}\right) \mathrm{A}_{\xi, a}^{(2)} \exp \left(-S^{(2)}\right)=\xi\left(\mathrm{J}_\xi^{(1)}\left(\begin{array}{c|c}E^{(2)} & 0 \\ \hline 0 & g^{(2)}\end{array}\right)+\mathcal{O}(a^6)\right),
\end{equation}
with $E^{(2)}$ and $g^{(2)}$ given in \eqref{6.10}.

\begin{proof}[Proof of Theorem \ref{the6.1}]The matrix $\mathrm{U}:=\xi \hat{\mathrm{J}}_\xi E^{(2)}$ admits the expansion specified in \eqref{6.25} and the number $\mathrm{g}:=\xi\left(i +r_7\left(\xi ^2\right)\right)g^{(2)}$ exhibits expansions detailed in \eqref{6.26}, where $\mathrm{e}_{33}$ is defined in \eqref{**}.
\end{proof}
The first conclusion of Theorem \ref{the1.1} is an immediate consequence of Theorem \ref{the6.1}, while the second requires the additional consideration of Remark \ref{rem3.1}.

\section{Spectral instability away from origin}\label{sec non-modulational}
In this section, we give the spectral analysis away from the origin, i.e., we consider the eigenvalue collisions $\omega_{0}:=\omega_{j, \xi_0}=\omega_{j-2, \xi_0}, j=0, 1$ for $k^2>3$ as described in Lemma \ref{lem3.3}.  When $k^2=4$, we observe the special case $\xi_0=\frac 1 2$. Since our normalization of the Floquet parameter $\xi\in(-\frac 1 2, \frac 1 2]$, perturbing in $\xi$ then requires to consider both $\xi$ near $\frac 1 2$ and $\xi$ near $-\frac 1 2$ (a positive perturbation from $\frac 1 2$ is mapped by subtracting 1 to a point near $-\frac 1 2$). This creates notational but not mathematical complexity. In fact, to avoid endpoint issues, one may take the fundamental domain as $\xi\in(0, \frac 1 2]$. We therefore restrict our proof to the case $\xi_0 \neq \frac 1 2$ (i.e., $\xi_0 \in(0, \frac 1 2)$) for simplicity.

The current analysis aligns closely with the analytical approaches developed in Sections \ref{sec basis} and \ref{sec Matrix}. To begin with, we set
$f_j^{+}(\xi,0)=\mathrm{e}^{i j z}, f_{j-2}^{-}(\xi,0)=\mathrm{e}^{i (j-2) z}.$
Then we extend those through
\[
f_j^{+}(\xi,a)=\bar{U}_{\xi, a}f_j^{+}(\xi,0), \qquad f_{j-2}^{-}(\xi,a)=\bar{U}_{\xi, a}f_{j-2}^{-}(\xi,0),
\]
where the extension operator $\bar{U}_{\xi, a}$ is obtained by solving \cite[Chapter Two, §4]{kato2013}
$$
\partial_a \bar{U}_{\xi, a}=\left[\partial_a \bar{P}_{\xi, a}, \bar{P}_{\xi, a}\right] \bar{U}_{\xi, a} \quad \bar{U}_{\xi, 0}=I,
$$
defined for $\varepsilon>0$ sufficiently small, $|a| \leq \varepsilon,\left|\xi-\xi_0\right| \leq \varepsilon$.
Let us also determine a dual basis spanning the sum of characteristic spaces of $\mathcal{A}_{\xi, a}^*$ associated with eigenvalues in $B_{\varepsilon}\left(\overline{\omega_0}\right)$. As direct computations give
\begin{equation}\label{7.1}
\begin{aligned}
&\left(\mathcal{E}_\xi f_j^{+}(\xi,0), f_j^{+}(\xi,0)\right)=-2i \frac {1+k^2(j+\xi)^2}{j+\xi}:=-2ip_j,\\
&\left(\mathcal{E}_\xi f_{j-2}^{-}(\xi,0), f_{j-2}^{-}(\xi,0)\right)=-2i \frac {1+k^2(j-2+\xi)^2}{j-2+\xi}:=-2ip_{j-2},\\
&\left(\mathcal{E}_\xi f_j^{+}(\xi,0), f_{j-2}^{-}(\xi,0)\right)=\left(\mathcal{E}_\xi f_{j-2}^{-}(\xi,0), f_j^{+}(\xi,0)\right)=0,
\end{aligned}
\end{equation}
Then using that $\bar{U}_{\xi, a}$ are symplectic, we get
\[
\left(\mathcal{E}_\xi f_j^\sigma(\xi,a), f_{j^{\prime}}^{\sigma^{\prime}}(\xi,a)\right)=
\left(\mathcal{E}_\xi \bar{U}_{\xi, a} f_j^\sigma(\xi,0), \bar{U}_{\xi, a} f_{j^{\prime}}^{\sigma^{\prime}}(\xi,0)\right)
=(\mathcal{E}_\xi f_j^\sigma(\xi,0), f_{j^{\prime}}^{\sigma^{\prime}}(\xi,0)).
\]
Hence a dual basis is given as
\[
\tilde{f}_j^{+}(\xi,a)= \frac {i}{2p_j}\mathcal{E}_\xi f_j^{+}(\xi,a),\quad
\tilde{f}_{j-2}^{-}(\xi,a)=\frac {i}{2p_{j-2}}\mathcal{E}_\xi f_{j-2}^{-}(\xi,a).
\]
Then the matrix representation of $\mathcal{A}_{\xi, a}$ on the spectral subspace associated with $\sigma (\mathcal{A}_{\xi, a})\cap B_{\varepsilon}\left(\omega_0\right)$ is
\[
\begin{aligned}
\mathrm{A}_{\xi, a}&=\left(\begin{array}{cc}
\left(\tilde{f}_j^{+}(\xi,a), \mathcal{A}_{\xi, a} f_j^{+}(\xi,a)\right) &
\left(\tilde{f}_{j-2}^{-}(\xi,a), \mathcal{A}_{\xi, a} f_{j}^{+}(\xi,a)\right) \\
\left(\tilde{f}_j^{+}(\xi,a), \mathcal{A}_{\xi, a} f_{j-2}^{-}(\xi,a)\right) & \left(\tilde{f}_{j-2}^{-}(\xi,a), \mathcal{A}_{\xi, a} f_{j-2}^{-}(\xi,a)\right)
\end{array}\right) \\
&=\frac i 2 \left(\begin{array}{cc}
\left(\frac {1}{p_j}\mathcal{L}_{\xi, a}f_j^{+}(\xi,a), f_j^{+}(\xi,a)\right) &
\left(\frac {1}{p_{j-2}}\mathcal{L}_{\xi, a}f_{j-2}^{-}(\xi,a), f_{j}^{+}(\xi,a)\right) \\
\left(\frac {1}{p_j}\mathcal{L}_{\xi, a}f_j^{+}(\xi,a),  f_{j-2}^{-}(\xi,a)\right) & \left(\frac {1}{p_{j-2}}\mathcal{L}_{\xi, a}f_{j-2}^{-}(\xi,a), f_{j-2}^{-}(\xi,a)\right)
\end{array}\right) \\
&:=\frac i 2 \left(\begin{array}{cc}
\frac {c_+(\xi,a)}{p_j} &
\frac {b(\xi,a)}{p_{j-2}} \\
\overline{\frac {b(\xi,a)}{p_j}} &
\frac {c_-(\xi,a)}{p_{j-2}}
\end{array}\right),
\end{aligned}
\]
where $c_+(\xi,a), c_-(\xi,a)\in\mr$ and $\left(\mathcal{L}_{\xi, a}f_{j-2}^{-}(\xi,a), f_{j}^{+}(\xi,a)\right)=\overline{\left(\mathcal{L}_{\xi, a}f_j^{+}(\xi,a),  f_{j-2}^{-}(\xi,a)\right)}$.

We calculate the eigenvalues of $\mathrm{A}_{\xi, a}$ as
\begin{equation}\label{7.2}
\frac{i}{2}\left(\frac{1}{2}\left(\frac{c_{+}(\xi, a)}{p_{j}(\xi)}+\frac{c_{-}(\xi, a)}{p_{j-2}(\xi)}\right) \pm \sqrt{\frac{1}{4}\left(\frac{c_{+}(\xi, a)}{p_{j}(\xi)}-\frac{c_{-}(\xi, a)}{p_{j-2}(\xi)}\right)^2+\frac{|b(\xi, a)|^2}{p_{j}(\xi) p_{j-2}(\xi)}}\right).
\end{equation}
Inspired by \cite{noble2023} and given $p_{j}(\xi) p_{j-2}(\xi)<0, j=0,1$, the possible emergence of unstable spectrum near $(\omega_0,\xi_0,0)$ is effectively reduced to the fact that $\frac{c_{+}(\xi, a)}{p_{j}(\xi)}-\frac{c_{-}(\xi, a)}{p_{j-2}(\xi)}$ could take zero. Note that
\[
\mathrm{A}_{\xi, 0}=\left(\begin{array}{cc}
i\frac{2 k^2 w_0^2}{1+k^2}\frac{j+\xi}{1+k^2 (j+\xi)^2}\left((j+\xi)^2-1\right) &
0 \\
0& i\frac{2 k^2 w_0^2}{1+k^2}\frac{j-2+\xi}{1+k^2 (j+\xi)^2}\left((j-2+\xi)^2-1\right)
\end{array}\right),
\]
then for $k^2>3$,
\[
\begin{aligned}
&\partial_\xi\left(\frac{c_{+}(\xi,0)}{2p_{j}(\xi)}-\frac{c_{-}(\xi, 0)}{2p_{j-2}(\xi)}\right)(\xi_0)\\
&=\frac{2 k^2 w_0^2}{1+k^2}\partial_\xi\left(\frac{j+\xi}{1+k^2 (j+\xi)^2}\left((j+\xi)^2-1\right)-\frac{j-2+\xi}{1+k^2 (j+\xi)^2}\left((j-2+\xi)^2-1\right)\right)(\xi_0)\\
&=\left\{\begin{aligned}
&\frac{\sqrt{1-\frac{3}{k^2}} \left(k^2-3\right)}{k^2+1}>0, \quad j=0,\\
&-\frac{\sqrt{1-\frac{3}{k^2}}\left(k^2-3\right)}{k^2+1}<0, \quad j=1.
\end{aligned}\right.
\end{aligned}
\]
By Implicit Function Theorem, there exists a smooth function $a\rightarrow \Xi(a)$ such that $\Xi(0)=\xi_0$ and, for any sufficiently small $|a|$,
$$
\left(\frac{c_{+}(\xi,a)}{2p_{j}(\xi)}-\frac{c_{-}(\xi, a)}{2p_{j-2}(\xi)}\right)(\Xi(a))=0.
$$
Finally, note that
\[
\begin{aligned}
f_{j-2}^{-}(\Xi(a),a)&=\mathrm{e}^{i (j-2) z}+\nu_1 \mathrm{e}^{i (j-3) z}a+\tilde{\nu}_1\mathrm{e}^{i (j-1) z}a+\mathcal{O}(a^2),\quad \nu_1, \tilde{\nu}_1\in \mathbb{C},\\
f_{j}^{+}(\Xi(a),a)&=\mathrm{e}^{i j z}+\nu_2 \mathrm{e}^{i (j+1) z}a+\tilde{\nu}_2\mathrm{e}^{i (j-1) z}a+\mathcal{O}(a^2), \quad \nu_2, \tilde{\nu}_2\in \mathbb{C},
\end{aligned}
\]
where the fact $\Xi(a)=\xi_0+\mathcal{O}(a)$ is concerned. Then, a straightforward calculation reveals  $b(\Xi(a),a)=\left(\mathcal{L}_{\Xi(a), a}f_{j-2}^{-}(\Xi(a),a), f_{j}^{+}(\Xi(a),a)\right)=\mathcal{O}(a^2)$ by the expression of $\mathcal{L}_{\xi, a}$ in \eqref{A_xi}. This completes the proof of (iii) in Theorem \ref{the1.1}.

\appendix
\section{Small-amplitude expansion}\label{app A}
In this section, we give the details on  small-amplitude expansion of \eqref{2.9}-\eqref{2.11}. Since $w$ and $c$ depend analytically on $a$ for $|a|$ sufficiently small and since $c$ is even in $a$, we write that
$$
w(k, a, b)(z):=w_0(k, b)+a \cos z+a^2 w_2(z)+a^3 w_3(z)+\mathcal{O}\left(a^4\right)
$$
and
$$
c(k, a, b):=c_0(k,b)+a^2 c_2+\mathcal{O}\left(a^4\right)
$$
as $a\rightarrow0$, where $w_2, w_3,$ . . . are even and $2\pi$-periodic in $z$. Substituting these into \eqref{2.3}, at the order of $a^2$, we gather that
\[
\frac{b \left(2 b k^2 w_2''(z)+2 b k^2 w_2(z)+\frac{1}{2} \left(k^2+1\right)^{2/3} \left(-2 c_2+3 \left(k^2+1\right) \cos (2 z)+k^2+3\right)\right)}{\sqrt[3]{k^2+1}}=0.
\]
Direct computation then reveals that
\begin{equation}\label{A.1}
w_2=\frac{\left(k^2+1\right)^{2/3} \left(2 c_2+\left(k^2+1\right) \cos (2 z)-k^2-3\right)}{4 b k^2}.
\end{equation}
At the order of $a^3$,
\[
\frac{2 b^2 k^2 \left(w_3''(z)+w_3(z)\right)}{\sqrt[3]{k^2+1}}+\frac{\left(k^2+1\right) \cos (z) \left(6 c_2-4 k^4-7 k^2-9+\left(k^2+1\right) \left(8 k^2+3\right) \cos (2 z)\right)}{2 k^2}=0.
\]
Noting that
\[
\cos 3\alpha=2\cos 2\alpha \cos\alpha-\cos\alpha,
\]
we choose
\begin{equation}\label{A.2}
c_2=\frac{1}{4} \left(k^2+5\right).
\end{equation}
Then \eqref{A.1} gives
\[
w_2=\frac{\left(k^2+1\right)^{5/3} (-1+2 \cos (2 z))}{8 b k^2}.
\]
\section{Expansion of the Kato basis}\label{app B}
First, we develop the Taylor expansion of the operators $U_{\xi, a}$ specified in \eqref{4.1}, employing prime notation for $a$-derivatives and dot notation for $\xi$-derivatives. The lemma below, whose proof follows arguments similar to \cite[Lemma A.1]{2022INVENT}, is stated without demonstration.
\begin{lemma}\label{lemB.1} The first jets of $U_{\xi, a} P_{0,0}$ are
\begin{align}
U_{0,0} P_{0,0} & =P_{0,0}, \quad U_{0,0}^{\prime} P_{0,0}=P_{0,0}^{\prime} P_{0,0}, \quad \dot{U}_{0,0} P_{0,0}=\dot{P}_{0,0} P_{0,0}, \label{B.1}\\
\dot{U}_{0,0}^{\prime} P_{0,0} & =\left(\dot{P}_{0,0}^{\prime}-\frac{1}{2} P_{0,0} \dot{P}_{0,0}^{\prime}\right) P_{0,0},\label{B.2}
\end{align}
where
\begin{align}
& P_{0,0}^{\prime}=\frac{1}{2 \pi i} \oint_{\Gamma}\left(\mathcal{A}_{0,0}-\lambda\right)^{-1} \mathcal{A}_{0,0}^{\prime}\left(\mathcal{A}_{0,0}-\lambda\right)^{-1} \mathrm{~d} \lambda, \label{B.3}\\
& \dot{P}_{0,0}=\frac{1}{2 \pi i} \oint_{\Gamma}\left(\mathcal{A}_{0,0}-\lambda\right)^{-1} \dot{\mathcal{A}}_{0,0}\left(\mathcal{A}_{0,0}-\lambda\right)^{-1} \mathrm{~d} \lambda,\label{B.4}
\end{align}
and
\begin{equation}\label{B.5}
\begin{aligned}
\dot{P}_{0,0}^{\prime}=- & \frac{1}{2 \pi i} \oint_{\Gamma}\left(\mathcal{A}_{0,0}-\lambda\right)^{-1} \dot{\mathcal{A}}_{0,0}\left(\mathcal{A}_{0,0}-\lambda\right)^{-1} \mathcal{A}_{0,0}^{\prime}\left(\mathcal{A}_{0,0}-\lambda\right)^{-1} \mathrm{~d} \lambda \\
& -\frac{1}{2 \pi i} \oint_{\Gamma}\left(\mathcal{A}_{0,0}-\lambda\right)^{-1} \mathcal{A}_{0,0}^{\prime}\left(\mathcal{A}_{0,0}-\lambda\right)^{-1} \dot{\mathcal{A}}_{0,0}\left(\mathcal{A}_{0,0}-\lambda\right)^{-1} \mathrm{~d} \lambda \\
& +\frac{1}{2 \pi i} \oint_{\Gamma}\left(\mathcal{A}_{0,0}-\lambda\right)^{-1} \dot{\mathcal{A}}_{0,0}^{\prime}\left(\mathcal{A}_{0,0}-\lambda\right)^{-1} \mathrm{~d} \lambda.
\end{aligned}
\end{equation}
The operators $\mathcal{A}_{0,0}^{\prime}$, $\dot{\mathcal{A}}_{0,0}$ and $\dot{\mathcal{A}^{\prime}}_{0,0}$ are
\begin{align}
\mathcal{A}_{0,0}^{\prime}=&\partial_z\left( 1-k^2\partial_z^2\right)^{-1}
\left(2k^2\partial_z\left(w_0\cos z \p_z\right)
-(6+2k^2)w_0\cos z\right), \label{B.6} \\
\dot{\mathcal{A}}_{0,0}=& i\left(c_0-w_0^2-2w_0^2(1+k^2\partial_z^2)(1-k^2\partial_z^2)^{-2} \right), \label{B.7}\\
\dot{\mathcal{A}^{\prime}}_{0,0}=&i2k^2 w_0\partial_z \left(1-k^2\partial_z^2\right)^{-1}\left(2\cos z \partial_z-\sin z\right)+i\left[\left(1-k^2\partial_z^2\right)^{-1} +2k^2\partial_z^2\left(1-k^2\partial_z^2\right)^{-2}\right]\nonumber\\
&\left(2k^2\partial_z\left(w_0\cos z \p_z\right)
-(6+2k^2)w_0\cos z\right).\label{B.8}
\end{align}
\end{lemma}
The following lemma characterizes the projectors $P_{\xi, a}$ and the transformation operators $U_{\xi, a}$ at $a=0$:
\begin{lemma}\label{lemB.2} For every $\xi$ small enough, one has $P_{\xi, 0} P_{0,0}=P_{0,0}$ and $U_{\xi, 0} P_{0,0}=P_{0,0}$. In particular $f_m^\sigma(\xi, 0)=f_m^\sigma$ for $(m, \sigma) \in\{(1, \pm),(0,+)\}$.
\end{lemma}
\begin{proof}
When $\xi$ is sufficiently small, the basis $\{f_1^{\pm},f_0^+\}$ is a basis of $\mathcal{V}_{\xi, 0}$, yielding $P_{\xi, 0}=P_{0,0}$. This result, combined with the operator definition \eqref{4.1}, establishes the relation $U_{\xi, 0} P_{0,0}=P_{0,0}$.
\end{proof}
The Lemma \ref{lemB.1} and \ref{lemB.2}  give the Taylor expansion
\begin{equation}\label{B.9}
f_m^\sigma(\xi, a)=f_m^\sigma+a P_{0,0}^{\prime} f_m^\sigma +\xi a\left(\dot{P}_{0,0}^{\prime}-\frac{1}{2} P_{0,0} \dot{P}_{0,0}^{\prime}\right) f_m^\sigma+\mathcal{O}\left(\xi^2 a, a^2\right), \quad (m, \sigma) \in\{(1, \pm),(0,+)\}.
\end{equation}
We observe that the term $\xi \dot{P}_{0,0} f_m^\sigma$ necessarily vanishes due to the identity $f_m^\sigma(\xi, 0) \equiv f_m^\sigma$. The computation of the vectors $P_{0,0}^{\prime} f_m^\sigma$ using \eqref{B.3} requires precise understanding of how $\left(\mathcal{A}_{0,0}-\lambda\right)^{-1}$ operates on the vectors
\begin{equation}\label{B.10}
\left(\begin{array}{c}
\cos (n z) \\
\sin (n z)
\end{array}\right), \quad n \in \mathbb{N}.
\end{equation}
We have the following result:
\begin{lemma}\label{lemB.3}
The space $H^{1}(\mathbb{T})$ decomposes as $H^{1}(\mathbb{T})=\mathcal{V}_{0,0} \oplus \mathcal{W}_{H^{1}}$ with $\mathcal{W}_{H^{1}}=\overline{\bigoplus_{n=2}^{\infty} \mathcal{W}_n}^{H^{1}}$ where the subspaces $\mathcal{V}_{0,0}$ and $\mathcal{W}_n$, defined below, are invariant under $\mathcal{A}_{0,0}$ and the following properties hold:
\begin{enumerate}[label=\textup{(\roman*)}]
\item $\mathcal{V}_{0,0}=\operatorname{span}\left\{f_1^{+}, f_1^{-}, f_0^{+}\right\}$ is the kernel of $\mathcal{A}_{0,0}$. For any $\lambda \neq 0$ the operator $\mathcal{A}_{0,0}-\lambda: \mathcal{V}_{0,0} \rightarrow \mathcal{V}_{0,0}$ is invertible and
\begin{equation}\label{B.11}
\left(\mathcal{A}_{0,0}-\lambda\right)^{-1} f_m^\sigma=-\frac{1}{\lambda} f_m^\sigma, \quad \forall(m, \sigma) \in\{(1, \pm),(0,+)\}.
\end{equation}
\item Each subspace $\mathcal{W}_n:=\operatorname{span}\left\{\cos (n z), \sin (nz)\right\}$ is invariant under $\mathcal{A}_{0,0}$. Let $\mathcal{W}_{L^2}=\overline{\bigoplus_{n=2}^{\infty} \mathcal{W}_n}^{L^2}$. For any $|\lambda|<\delta_0$ small enough, the operator $\mathcal{A}_{0,0}-\lambda: \mathcal{W}_{H^{1}} \rightarrow \mathcal{W}_{L^2}$ is invertible and in particular
\begin{equation}\label{C.12}
\begin{aligned}
\left(\mathcal{A}_{0,0}-\lambda\right)^{-1} \left(\begin{array}{c}
\cos (n z) \\
\sin (n z)
\end{array}\right)&= \frac{(1+k^2)(1+k^2n^2)}{2 k^2 w_0^2 n (n^2-1)}\left(\begin{array}{c}
\sin (n z) \\
-\cos (n z)
\end{array}\right)\\
&-\lambda\left(\frac{(1+k^2)(1+k^2n^2)}{2 k^2 w_0^2 n (n^2-1)}\right)^2 \left(\begin{array}{c}
\cos (n z) \\
\sin (n z)
\end{array}\right)
+\mathcal{O}(\lambda^2).
\end{aligned}
\end{equation}
\end{enumerate}
\end{lemma}
\begin{proof}(i) For $f \in \mathcal{V}_{0,0}=\operatorname{ker}\left(\mathcal{A}_{0,0}\right)$ and $\lambda \neq 0$ we have that $\left(\mathcal{A}_{0,0}-\lambda\right) f=-\lambda f$. Inverting both sides of the equation gives \eqref{B.11}.

(ii) The invariance of $\mathcal{W}_n$ under $\mathcal{A}_{0,0}$ is immediate. The corresponding matrix representation of $\mathcal{A}_{0,0}: \mathcal{W}_n \rightarrow \mathcal{W}_n$ in the basis $\left\{\cos (n z), \sin (nz)\right\}$ takes the form
\begin{equation}\label{B.13}
\mathrm{A}_{0,0}=\left(\begin{array}{cc}
0 & -\frac{n(n^2-1)}{1+k^2 n^2}\frac{2 k^2 w_0^2}{1+k^2} \\
\frac{n(n^2-1)}{1+k^2 n^2}\frac{2 k^2 w_0^2}{1+k^2} & 0
\end{array}\right).
\end{equation}
As $n \geq 2 $, so the matrix $\mathrm{A}_{0,0}$ is invertible with inverse
\begin{equation}\label{C.14}
\mathrm{A}_{0,0}^{-1}=\left(\begin{array}{cc}
0 & \frac{(1+k^2)(1+k^2n^2)}{2 k^2 w_0^2 n (n^2-1)} \\
-\frac{(1+k^2)(1+k^2n^2)}{2 k^2 w_0^2 k^2 n (n^2-1)} & 0
\end{array}\right).
\end{equation}
The invertibility of $\mathcal{A}_{0,0}-\lambda$ and the formulas in \eqref{C.12} follow using the Neumann series expansion
\begin{equation}\label{C.15}
\left(\mathcal{A}_{0,0}-\lambda\right)^{-1}=\mathcal{A}_{0,0}^{-1}\left(1-\lambda \mathcal{A}_{0,0}^{-1}\right)^{-1}=\sum_{j=0}^{\infty} \lambda^j \mathcal{A}_{0,0}^{-j-1}=\mathcal{A}_{0,0}^{-1}+\sum_{j=0}^{\infty} \lambda^{j+1} \mathcal{A}_{0,0}^{-j-2}
\end{equation}
which converges provided $|\lambda|$ is small enough.
\end{proof}
We now proceed to compute the action of the operator $\left(\mathcal{A}_{0,0}-\lambda\right)^{-1} \mathcal{A}_{0,0}^{\prime}$ on the vectors $f_1^{ \pm}, f_0^{+}$.
\begin{lemma}[Action of $\left(\mathcal{A}_{0,0}-\lambda\right)^{-1} \mathcal{A}_{0,0}^{\prime}$ on $\left(\mathcal{V}_{0,0}\right)$]\label{lemB.4}
One has
\begin{equation}\label{B.16}
\begin{aligned}
\left(\mathcal{A}_{0,0}-\lambda\right)^{-1} \mathcal{A}_{0,0}^{\prime} \left(\begin{array}{c}
\cos z \\
\sin z
\end{array}\right)=&-\frac {(k^2+1)^2} {2 k^2 w_0} \left(\begin{array}{c}
\cos (2 z) \\
\sin (2 z)
\end{array}\right)\\
&-\lambda \frac {(k^2+1)^3 (1+4k^2)} {24 k^4 w_0^3} \left(\begin{array}{c}
\sin(2 z) \\
-\cos (2 z)
\end{array}\right)
+\mathcal{O}_{\mathcal{W}}\left(\lambda^2\right), \\
\left(\mathcal{A}_{0,0}-\lambda\right)^{-1} \mathcal{A}_{0,0}^{\prime} 1=&-\frac{(6+2k^2)w_0}{\lambda (1+k^2)} \sin z.
\end{aligned}
\end{equation}
\end{lemma}
\begin{proof} Recall \eqref{B.6}, and so
\begin{equation*}
\begin{aligned}
\mathcal{A}_{0,0}^{\prime} \left(\begin{array}{c}
\cos z \\
\sin z
\end{array}\right)=&-\partial_z\left( 1-k^2\partial_z^2\right)^{-1} \left(\begin{array}{c}
-2k^2w_0\cos(2z)-(3+k^2)w_0\left(\cos(2z)+1\right) \\
-2k^2w_0\sin(2z)-(3+k^2)w_0\sin(2z)
\end{array}\right)\\
=&-3(1+k^2)w_0\partial_z\left( 1-k^2\partial_z^2\right)^{-1}\left(
\begin{array}{c}
\cos(2z)\\
\sin(2z)
\end{array}
\right)\\
=&\frac {6(k^2+1)} {1+4k^2}w_0\left(
\begin{array}{c}
\sin(2z)\\
-\cos(2z)
\end{array}
\right),
\end{aligned}
\end{equation*}
and
\begin{equation*}
\mathcal{A}_{0,0}^{\prime} 1=-(6+2k^2)w_0\partial_z\left( 1-k^2\partial_z^2\right)^{-1}
\left(\cos z\right)=\frac{6+2k^2}{1+k^2}w_0\sin z.
\end{equation*}
Applying Lemma \ref{B.3} gives the result.
\end{proof}
We are now prepared to compute the terms $P_{0,0}^{\prime} f_m^\sigma$ in \eqref{B.9}:
\begin{lemma}\label{lemB.5} One has
\begin{equation}\label{B.17}
P_{0,0}^{\prime} \left(\begin{array}{c}
\cos z \\
\sin z
\end{array}\right)=2A_2\left(
\begin{array}{c}
\cos(2z)\\
\sin(2z)
\end{array}\right), \quad
P_{0,0}^{\prime} 1=0.
\end{equation}
\end{lemma}
\begin{proof}
Utilizing equation \eqref{B.3} and Lemma \ref{lemB.3}-(i), one deduces that
$$
P_{0,0}^{\prime} \left(\begin{array}{c}
\cos z \\
\sin z
\end{array}\right)=-\frac{1}{2 \pi i} \oint_{\Gamma} \frac{1}{\lambda}\left(\mathcal{A}_{0,0}-\lambda\right)^{-1} \mathcal{A}_{0,0}^{\prime} \left(\begin{array}{c}
\cos z \\
\sin z
\end{array}\right) \mathrm{d} \lambda .
$$
Application of the residue theorem reveals that only the $O(1)$ terms in $\lambda$ in the expansion of $-\left(\mathcal{A}_{0,0}-\lambda\right)^{-1} \mathcal{A}_{0,0}^{\prime} f_m^\sigma$ yield non-vanishing contributions. Then the first equality follows by Lemma \ref{lemB.4} and \eqref{2.11}. Moreover
$$
P_{0,0}^{\prime}1=\frac{6+2k^2}{1+k^2}w_0 \frac{1}{2 \pi i} \oint_{\Gamma} \frac{1}{\lambda^2}\mathrm{d} \lambda=0.
$$
This completes the proof.
\end{proof}
\begin{lemma}[Action of $\left(\mathcal{A}_{0,0}-\lambda\right)^{-1} \dot{\mathcal{A}}_{0,0}$ on $\mathcal{V}_{0,0}$]\label{lemB.6}
One has
\begin{equation}\label{B.18}
\begin{aligned}
&\left(\mathcal{A}_{0,0}-\lambda\right)^{-1} \dot{\mathcal{A}}_{0,0} \left(\begin{array}{c}
\cos z \\
\sin z
\end{array}\right)=-i \frac{4 k^2 w_0^2}{\lambda(1+k^2)^2} \left(\begin{array}{c}
\cos z \\
\sin z
\end{array}\right), \\
&\left(\mathcal{A}_{0,0}-\lambda\right)^{-1} \dot{\mathcal{A}}_{0,0} 1=i \frac{2 k^2 w_0^2}{\lambda(1+k^2)},
\end{aligned}
\end{equation}
and, for $n \geq 2$,
\begin{equation}\label{B.19}
\left(\mathcal{A}_{0,0}-\lambda\right)^{-1} \dot{\mathcal{A}}_{0,0} \left(\begin{array}{c}
\cos (n z) \\
\sin (n z)
\end{array}\right)=i \frac{k^2n^2(n^2+1)+3 n^2-1}{n \left(n^2-1\right) \left(k^2 n^2+1\right)} \left(\begin{array}{c}
\sin (n z) \\
-\cos (n z)
\end{array}\right)+\mathcal{O}_{\mathcal{W}}(\lambda) .
\end{equation}
\end{lemma}
\begin{proof} By \eqref{B.7}, one has
$$
\dot{\mathcal{A}}_{0,0} \left(\begin{array}{c}
\cos z \\
\sin z
\end{array}\right)=i\left(c_0-w_0^2-2w_0^2 \frac{(1-k^2)}{(1+k^2)^2}\right) \left(\begin{array}{c}
\cos z \\
\sin z
\end{array}\right)=i \frac{4 w_0^2 k^2}{(1+k^2)^2}\left(\begin{array}{c}
\cos z \\
\sin z
\end{array}\right),
$$
and
\[
\dot{\mathcal{A}}_{0,0} 1 =i\left(c_0-3 w_0^2\right)=-i \frac{2 k^2 w_0^2}{1+k^2}.
\]
Then \eqref{B.18} follows using \eqref{B.11}. Formula \eqref{B.19} is proved on the same lines, as
$$
\begin{aligned}
\dot{\mathcal{A}}_{0,0} \left(\begin{array}{c}
\cos (n z) \\
\sin (n z)
\end{array}\right)&=i\left(c_0-w_0^2-2w_0^2 \frac{(1-k^2n^2)}{(1+k^2n^2)^2}\right)  \left(\begin{array}{c}
\cos (n z) \\
\sin (n z)
\end{array}\right)\\
&=i 2 k^2 w_0^2 \frac{k^2n^2(n^2+1)+3 n^2-1}{(1+k^2)(1+k^2n^2)^2}\left(\begin{array}{c}
\cos (n z) \\
\sin (n z)
\end{array}\right).
\end{aligned}
$$
Employing \eqref{C.12} gives \eqref{B.19}.
\end{proof}
\begin{lemma}\label{lemB.7}
One has
\begin{equation}\label{B.20}
\dot{P}_{0,0}^{\prime} \left(\begin{array}{c}
\cos z \\
\sin z
\end{array}\right)=2 i A_2 \left(\begin{array}{c}
-\sin(2 z) \\
\cos (2 z)
\end{array}\right), \quad \dot{P}_{0,0}^{\prime} 1=0.
\end{equation}
Consequently,
\begin{equation}\label{B.21}
\left(\dot{P}_{0,0}^{\prime}-\frac{1}{2} P_{0,0} \dot{P}_{0,0}^{\prime}\right) \left(\begin{array}{c}
\cos z \\
\sin z
\end{array}\right)=2 i A_2 \left(\begin{array}{c}
-\sin(2 z) \\
\cos (2 z)
\end{array}\right),\quad
\left(\dot{P}_{0,0}^{\prime}-\frac{1}{2} P_{0,0} \dot{P}_{0,0}^{\prime}\right) 1=0.
\end{equation}
\end{lemma}
\begin{proof}By \eqref{B.5} and \eqref{B.11} we write,
$$
\begin{aligned}
\dot{P}_{0,0}^{\prime} \left(\begin{array}{c}
\cos z \\
\sin z
\end{array}\right)= & \frac{1}{2 \pi i} \oint_{\Gamma} \frac{1}{\lambda}\left(\mathcal{A}_{0,0}-\lambda\right)^{-1} \dot{\mathcal{A}}_{0,0}\left(\mathcal{A}_{0,0}-\lambda\right)^{-1} \mathcal{A}_{0,0}^{\prime} \left(\begin{array}{c}
\cos z \\
\sin z
\end{array}\right) \\
&+\frac{1}{2 \pi i} \oint_{\Gamma} \frac{1}{\lambda}\left(\mathcal{A}_{0,0}-\lambda\right)^{-1} \mathcal{A}_{0,0}^{\prime}\left(\mathcal{A}_{0,0}-\lambda\right)^{-1} \dot{\mathcal{A}}_{0,0} \left(\begin{array}{c}
\cos z \\
\sin z
\end{array}\right) \mathrm{d} \lambda \\
& -\frac{1}{2 \pi i} \oint_{\Gamma} \frac{1}{\lambda}\left(\mathcal{A}_{0,0}-\lambda\right)^{-1} \dot{\mathcal{A}}_{0,0}^{\prime} \left(\begin{array}{c}
\cos z \\
\sin z
\end{array}\right) \mathrm{d} \lambda=: \mathrm{I}+\mathrm{II}+\mathrm{III}.
\end{aligned}
$$
Using \eqref{B.11}, \eqref{B.16}, \eqref{B.19} and the residue theorem we obtain
\begin{equation}\label{B.22}
\begin{aligned}
\mathrm{I}&=\frac{1}{2 \pi i} \oint_{\Gamma} \frac{1}{\lambda}\left(\mathcal{A}_{0,0}-\lambda\right)^{-1} \dot{\mathcal{A}}_{0,0}\left(-\frac {(k^2+1)^2} {2 k^2 w_0} \left(\begin{array}{c}
\cos (2 z) \\
\sin (2 z)
\end{array}\right)+\mathcal{O}_{\mathcal{W}}(\lambda)\right)\mathrm{d} \lambda \\
&=-i\frac{\left(k^2+1\right)^2 \left(20 k^2+11\right)}{12 w_0  k^2 \left(4 k^2+1\right)} \left(\begin{array}{c}
\sin (2 z) \\
-\cos (2 z)
\end{array}\right).
\end{aligned}
\end{equation}

Similarly, using \eqref{B.11}, \eqref{B.16}, \eqref{B.18} and the residue theorem gives
\begin{equation}\label{B.23}
\begin{aligned}
\mathrm{II}& =\frac{1}{2 \pi i} \oint_{\Gamma} \frac{1}{\lambda}\left(\mathcal{A}_{0,0}-\lambda\right)^{-1} \mathcal{A}_{0,0}^{\prime}\left(-i \frac{4 k^2 w_0^2}{\lambda(1+k^2)^2} \left(\begin{array}{c}
\cos z \\
\sin z
\end{array}\right)\right) \mathrm{d} \lambda \\
& =\left(i \frac{4 k^2 w_0^2}{(1+k^2)^2}\frac {(k^2+1)^3 (1+4k^2)} {24 k^4 w_0^3}\right)\frac{1}{2 \pi i} \oint_{\Gamma} \frac{1}{\lambda}  \left(\begin{array}{c}
\sin(2 z) \\
-\cos (2 z)
\end{array}\right) \mathrm{d} \lambda \\
& =i \frac{\left(k^2+1\right) \left(4 k^2+1\right)}{6 k^2 w_0} \left(\begin{array}{c}
\sin(2 z) \\
-\cos (2 z)
\end{array}\right).
\end{aligned}
\end{equation}
The final consideration is $\mathrm{III}$. By \eqref{B.8} we have
\begin{equation*}
\begin{aligned}
&\dot{\mathcal{A}}_{0,0}^{\prime} \left(\begin{array}{c}
\cos z \\
\sin z
\end{array}\right)\\
& =-i k^2 w_0\partial_z \left(1-k^2\partial_z^2\right)^{-1}\left(
\begin{array}{c}
3\sin(2z)\\
-1-3\cos(2z)
\end{array}\right)\\
&+i\left[\left(1-k^2\partial_z^2\right)^{-1} +2k^2\partial_z^2\left(1-k^2\partial_z^2\right)^{-2}\right]\left(\begin{array}{c}
-2k^2w_0\cos(2z)-(3+k^2)w_0\left(\cos(2z)+1\right) \\
-2k^2w_0\sin(2z)-(3+k^2)w_0\sin(2z)
\end{array}\right) \\
& =-i 3 k^2 w_0\partial_z \left(1-k^2\partial_z^2\right)^{-1}\left(
\begin{array}{c}
\sin(2z)\\
-\cos(2z)
\end{array}\right)\\
&+i\left[\left(1-k^2\partial_z^2\right)^{-1} +2k^2\partial_z^2\left(1-k^2\partial_z^2\right)^{-2}\right]
\left(-3(1+k^2)w_0 \left(\begin{array}{c}
\cos(2 z) \\
\sin (2 z)
\end{array}\right)+\left(\begin{array}{c}
-(3+k^2)w_0  \\
0
\end{array}\right)\right)\\
&=i\left(-\frac{6 k^2w_0}{1+4k^2}+3(1+k^2)w_0\left(\frac 1 {1+4k^2}-\frac{8k^2}{(1+4k^2)^2}\right) \right)\left(\begin{array}{c}
\cos(2 z) \\
\sin (2 z)
\end{array}\right)-i\left(\begin{array}{c}
(3+k^2)w_0 \\
0
\end{array}\right)\\
& =-i \frac{3w_0\left(12 k^4+5 k^2-1\right)}{\left(4 k^2+1\right)^2} \left(\begin{array}{c}
\cos(2 z) \\
\sin (2 z)
\end{array}\right)+\left(\begin{array}{c}
-i(3+k^2)w_0  \\
0
\end{array}\right),
\end{aligned}
\end{equation*}
so using also \eqref{B.11} we reach that
\begin{equation}\label{B.24}
\begin{aligned}
\mathrm{III}& =\frac{1}{2 \pi i} \oint_{\Gamma}\left(-\frac{1}{\lambda}\right)\left(\mathcal{A}_{0,0}-\lambda\right)^{-1} \dot{\mathcal{A}}_{0,0}^{\prime} \left(\begin{array}{c}
\cos z \\
\sin z
\end{array}\right) \mathrm{d} \lambda \\
& =\left(-i \frac{3w_0\left(12 k^4+5 k^2-1\right)}{\left(4 k^2+1\right)^2}\right)\frac{1}{2 \pi i} \oint_{\Gamma} \left(-\frac{1}{\lambda}\right)\left(\mathcal{A}_{0,0}-\lambda\right)^{-1}   \left(\begin{array}{c}
\cos(2 z) \\
\sin(2 z)
\end{array}\right) \mathrm{d} \lambda \\
& =i \frac{3w_0\left(12 k^4+5 k^2-1\right)}{\left(4 k^2+1\right)^2}\frac{\left(1+k^2\right) \left(1+4k^2\right)}{12 k^2 w_0^2} \left(\begin{array}{c}
\sin(2 z) \\
-\cos (2 z)
\end{array}\right)\\
& =-i \frac{\left(1+k^2\right) \left(12 k^4+5 k^2-1\right)}{4 k^2 w_0\left(4 k^2+1\right)} \left(\begin{array}{c}
\sin(2 z) \\
-\cos (2 z)
\end{array}\right).
\end{aligned}
\end{equation}
Summing \eqref{B.22}, \eqref{B.23}, \eqref{B.24} up we deduce that
\begin{equation*}
\begin{aligned}
&\mathrm{I}+\mathrm{II}+\mathrm{III}\\
&=i\left(-\frac{\left(k^2+1\right)^2 \left(20 k^2+11\right)}{12 w_0  k^2 \left(4 k^2+1\right)}+\frac{\left(k^2+1\right) \left(4 k^2+1\right)}{6 k^2 w_0}
-\frac{\left(1+k^2\right) \left(12 k^4+5 k^2-1\right)}{4 k^2 w_0\left(4 k^2+1\right)}\right)\left(\begin{array}{c}
\sin(2 z) \\
-\cos (2 z)
\end{array}\right)\\
&=i\frac{\left(k^2+1\right)^2}{2 k^2 w_0}\left(\begin{array}{c}
-\sin(2 z) \\
\cos (2 z)
\end{array}\right),
\end{aligned}
\end{equation*}
which gives first of \eqref{B.20} by \eqref{2.11}.

We now prove the second of \eqref{B.20}. Similarly, we denote
\[
\dot{P}_{0,0}^{\prime} 1=\widehat{\mathrm{I}}+\widehat{\mathrm{II}}+\widehat{\mathrm{III}},
\]
where by \eqref{B.16} and \eqref{B.18} and the residue theorem, we have
\begin{equation*}
\begin{aligned}
\widehat{\mathrm{I}}=&\frac{1}{2 \pi i} \oint_{\Gamma} \frac{1}{\lambda}\left(\mathcal{A}_{0,0}-\lambda\right)^{-1} \dot{\mathcal{A}}_{0,0}\left(\mathcal{A}_{0,0}-\lambda\right)^{-1} \mathcal{A}_{0,0}^{\prime} 1
\mathrm{d} \lambda\\
=&\frac{1}{2 \pi i} \oint_{\Gamma} \left(\mathcal{A}_{0,0}-\lambda\right)^{-1} \dot{\mathcal{A}}_{0,0}\left(-\frac{(6+k^2)w_0}{\lambda^2 (1+k^2)} \sin z\right)\mathrm{d} \lambda\\
=&\frac{1}{2 \pi i} \oint_{\Gamma} \left(i\frac{4 k^2 w_0^3 (6+k^2)}{\lambda^3 (1+k^2)^3} \sin z\right)\mathrm{d} \lambda=0,
\end{aligned}
\end{equation*}
\begin{equation*}
\begin{aligned}
\widehat{\mathrm{II}}=&\frac{1}{2 \pi i} \oint_{\Gamma} \frac{1}{\lambda}\left(\mathcal{A}_{0,0}-\lambda\right)^{-1} \mathcal{A}_{0,0}^{\prime}\left(\mathcal{A}_{0,0}-\lambda\right)^{-1} \dot{\mathcal{A}}_{0,0}  1
\mathrm{d} \lambda\\
=&\frac{1}{2 \pi i} \oint_{\Gamma}\left(\mathcal{A}_{0,0}-\lambda\right)^{-1} \mathcal{A}_{0,0}^{\prime}\left(i \frac{2 k^2 w_0^2}{\lambda^2(1+k^2)}\right)\mathrm{d} \lambda\\
=&\frac{1}{2 \pi i} \oint_{\Gamma} \left(-i\frac{2 k^2 w_0^3 (6+2k^2)}{\lambda^3 (1+k^2)^2} \sin z\right)\mathrm{d} \lambda=0,
\end{aligned}
\end{equation*}
Moreover since $\dot{\mathcal{A}}_{0,0}^{\prime} 1=i \hat{\alpha} \cos z$ for some constant $\hat{\alpha}$, by \eqref{B.11} we find also
$$
\widehat{\mathrm{III}}=-\frac{1}{2 \pi i} \oint_{\Gamma} \frac{1}{\lambda}\left(\mathcal{A}_{0,0}-\lambda\right)^{-1} \dot{\mathcal{A}}_{0,0}^{\prime} 1 \mathrm{d} \lambda=\frac{1}{2 \pi \mathrm{i}} \oint_{\Gamma} \frac{i \hat{\alpha}}{\lambda^2} \mathrm{d} \lambda=0.
$$

The qualities in \eqref{B.21} follow immediately from the observation $P_{0,0} \dot{P}_{0,0}^{\prime}=0$ as a result of $\dot{P}_{0,0}^{\prime} \mathcal{V}_{0,0} \subseteq \mathcal{W}_{L^2}$.
\end{proof}

Finally, an argument analogous to those in \cite{2022INVENT,ma2024} establishes Lemma \ref{lem5.1}, which we present here for completeness.

\begin{proof}[Proof of Lemma \ref{lem5.1}] It follows from the expansion in equation \eqref{B.9}, Lemma \ref{lemB.5} and Lemma \ref{lemB.7}. Using that
$$
\left(f_1^{+}(0, a), 1\right)=\left(U_{0, a} f_1^{+}, 1\right)=\left(U_{0, a} \mathcal{J}_0 f_1^{-}, 1\right)=\left(\mathcal{J}_0 U_{0, a}^{-*} f_1^{-}, 1\right)=0,
$$
one can prove that the term of order $a^2$ of $f_1^{+}(\xi, a)$ has zero average. Moreover, by the Taylor expansion, we get that
$$
\begin{aligned}
h(\xi, a):=&f_m^\sigma(\xi, a)-f_m^\sigma-\xi \dot{U}_{0,0} f_m^\sigma-a U_{0,0}^{\prime} f_m^\sigma-\xi a\dot{U}_{0,0}^{\prime} f_m^\sigma-\frac{1}{2} a^2 U_{0,0}^{\prime \prime} f_m^\sigma\\
=&\xi^3 \varphi_0(\xi, a)+\xi^2 a \varphi_1(\xi, a)+\xi a^2 \varphi_2(\xi, a)+a^3 \varphi_3(\xi, a)
\end{aligned}
$$
for some $\mathcal{C}^{r-2}$ functions $\varphi_j, j=0 \ldots 3$ and $r\geq 3$. By Lemma \ref{B.2}, $h(\xi, 0) \equiv 0$, and thus $\varphi_0(\xi, 0) \equiv 0$. Then $\varphi_0(\xi, a)=a \tilde{\varphi}_0(\xi, a)$ for some $\mathcal{C}^{r-3}$ function $\tilde{\varphi}_0(\xi, a)$ (\cite[Lemma B.1]{ma2024}), and in conclusion $h(\xi, a)=\mathcal{O}^r\left(\xi^2 a, \xi a^2, a^3\right)$. This gives that the remainders of the expansion are of order $\mathcal{O}\left(\xi^2 a, \xi a^2, a^3\right)$.
\end{proof}

\vspace{0.5cm}
\noindent {\bf Acknowledgements}
The work of LF was partially supported by NSFC Grant No. 12426607 and the NSF of Henan Province of China Grant No. 252300421218, the work of HG was supported by the Jiangsu Provincial Scientific Research Center of Applied Mathematics under Grant No. BK20233002, and the work of JL was partially supported by National Key R\&D Program of China (Grant No. 2024YFA1012801), and by National Natural Science Foundation of China (Grant No. 12571174 and 12426618).

\vspace{0.5cm}
\noindent {\bf Conflict of interest}
The authors declare that there is no conflict of interest.

\vspace{0.5cm}
\noindent {\bf Data Availability}
There is no data associated to this work.

\bibliographystyle{siam}
\bibliography{bKP}

\end{document}